\documentclass[11pt,a4paper]{article}
\usepackage{caption}
\usepackage{epsf,epsfig,amsfonts,amsgen,amsmath,amstext,amsbsy,amsopn,amsthm}
\usepackage{color}
\usepackage{mathrsfs}
\usepackage{tikz}
\usepackage{enumerate}
\usepackage{url}
\usepackage{enumitem}
\usepackage{appendix}
\usepackage{booktabs}
\usepackage{setspace}
\usepackage{comment}
\usepackage{pict2e,color}
\usepackage{tikz}
\usepackage{tkz-graph}
\usetikzlibrary{arrows,shapes,chains}
\usepackage {verbatim}
\usepackage{tikz-cd}
\usepackage{tablefootnote}

\usepackage{subfigure}
\usepackage[english]{babel}
\usepackage{url}

\newtheorem{thm}{Theorem}

\newtheorem{lem}[thm]{Lemma}

\newtheorem{prop}[thm]{Proposition}

\newtheorem{prob}[thm]{Problem}
\newtheorem{myDef}[thm]{Definition}

\newtheorem{gad}{Good Gadget}

\newcommand{\cp}{\vec{\operatorname{\bf{cp}}}}

\newcommand{\bfS}{\mathbf{S}}
\newcommand{\bfi}{{\bf C_i}}
\newcommand{\bfj}{{\bf C_j}}
\newcommand{\bfl}{{\bf C_\ell}}
\newcommand{\bfm}{{\bf C_m}}

\newcommand{\cH}{\mathcal{H}}

\begin{document}

\title{Discrepancies of perfect matchings in hypergraphs}
\author{
Hongliang Lu\footnote{School of Mathematics and Statistics, Xi'an Jiaotong University, Xi'an, Shaanxi 710049. Research supported by National Key Research and Development Program of China 2023YFA1010200. Email: luhongliang@mail.xjtu.edu.cn}~~~~~~~
Jie Ma\footnote{School of Mathematical Sciences, University of Science and Technology of China, Hefei, Anhui, 230026, China. Research supported by National Key Research and Development Program of China 2023YFA1010200 and National Natural Science Foundation of China grant 12125106. Email: jiema@ustc.edu.cn.}~~~~~~~
Shengjie Xie\footnote{School of Mathematical Sciences, University of Science and Technology of China, Hefei, Anhui, 230026, China. Email: jeff\_776532@mail.ustc.edu.cn.}
}

\date{}

\maketitle

\begin{abstract}
In this paper, we determine the minimum degree threshold of perfect matchings with high discrepancy in $r$-edge-colored $k$-uniform hypergraphs for all $k\geq 3$ and $r\geq 2$, thereby completing the investigation into discrepancies of perfect matchings that has recently attracted significant attention.
Our approach identifies this discrepancy threshold with a novel family of multicolored uniform hypergraphs and reveals new phenomena not covered in previous studies.
In particular, our results address a question of Balogh, Treglown and Z\'arate-Guer\'en concerning 3-uniform hypergraphs.
\end{abstract}



\section{Introduction}	
A {\it hypergraph} $H$ consists of a vertex set $V(H)$ and an edge set $E(H)$ whose members are subsets of $V(H)$.
For a positive integer $k$ and a set $S$, let $\binom{S}{k}:=\{T\subseteq S: |T|=k\}$.
A hypergraph $H$ is {\it $k$-uniform} if $E(H)\subseteq \binom{V(H)}{k}$,
and we often refer to a $k$-uniform hypergraph as a {\it $k$-graph}.
Let $H$ be a $k$-graph and $T\subseteq V(H)$.
The {\it neighbourhood} $N_H(T)$ of $T$ in $H$ denotes the family of subsets $S\subseteq V(H)\backslash T$ such that $T\cup S\in E(H)$.
When $T=\{x\}$, we express $N_{H}(\{x\})$ by $N_H(x)$.
The {\it degree} $d_H(T)$ of $T$ in $H$ is defined as $d_H(T)=|N_H(T)|$.
For an integer $\ell$ with $1\leq \ell\leq k-1$, let $\delta_\ell(H)=
\min\{d_H(T): T\in \binom{V(H)}{\ell}\}$ denote the {\it minimum $\ell$-degree} of $H$.
For simplicity, we often write $\delta_1(H)$ as $\delta(H)$ and term it the {\it minimum vertex degree} of $H$.
Throughout this paper, we often identify $E(H)$ with $H$ when there is no confusion.
A \emph{matching} in a hypergraph $H$ denotes a set of pairwise disjoint edges in $H$.
A matching in $H$  is {\it perfect} if it covers all vertices of $H$.

\subsection{Background}
The study of {\it discrepancy} of hypergraphs $\cH$ is a fundamental subject in combinatorics,
with the aim of estimating the maximum imbalance guaranteed to occur on some edge $e\in E(\cH)$ in every 2-coloring of $V(\cH)$.
For historical context, we reference Chapter 13 of \cite{AS} and Chapter 4 of \cite{M99}.
A natural multicolor extension can be formally described as follows.
For any integer $r\geq 2$, let $\cH$ be a hypergraph and $f: V(\cH)\rightarrow [r]$ be a $r$-coloring of its vertices.
For any color $c\in [r]$ and any edge $e\in E(\cH)$, let $c(e)$ denote the number of vertices in $e$ colored with $c$ under $f$,
and let $D_f(e)=\max_{c\in [r]} \left(c(e)-\frac{|e|}{r}\right)$.
The {\it $r$-color discrepancy} of a hypergraph $\cH$ is then defined as $D_r(\cH)=\min_{f:V(\cH)\to [r]}\max_{e\in E(\cH)} D_f(e).$
Therefore, the $2$-color discrepancy aligns with the initial definition of the discrepancy.

Initially introduced by Erd\H{o}s in the 1960s,
a well-explored discrepancy problem in the context of graphs considers $V(\cH)$ as the edge set of a graph $G$,
with $E(\cH)$ representing a collection of subgraphs of $G$ possessing specific properties such as spanning trees, Hamilton cycles and perfect matchings.
Two notable early results include Erd\H{o}s-Spencer \cite{ES71} on the discrepancy of cliques and Erd\H{o}s-F\"uredi-Loebl-S\'os \cite{EFLS} on the discrepancy of a given spanning tree in complete graphs.
In 2020, Balogh-Csaba-Jing-Pluh\'ar \cite{BCJP20} revisited this problem for general graphs $G$,
and since then, there has been extensive research investigating the $r$-color discrepancy of various properties of subgraphs in both graphs and hypergraphs, including \cite{BCPT,BTZ,Brad,BML,FHLT,FL,GGS,GKM1,GKM2,GKM3,MB23} in the literature.

In these recent studies, a significant focus has been on exploring the minimum ($\ell$-)degree threshold for perfect matchings with high discrepancy in graphs and hypergraphs.\footnote{Here and throughout, high discrepancy in a hypergraph $H$ always means a discrepancy with linear size $\Omega(|V(H)|)$.}
This exploration stems from the original problem of determining the minimum $\ell$-degree threshold for the existence of perfect matchings in uniform hypergraphs, a subject that has garnered considerable interest in recent decades.
Formally speaking, we are given integers $k,\ell,n$ with $1\leq \ell\leq k-1$ and $n\equiv0\pmod k$.
Let $m_\ell(k, n)$ denote the smallest integer $m$ such that every $n$-vertex $k$-graph $H$ with $\delta_\ell(H)\geq m$ contains a perfect matching.
We have the following general asymptotic lower bound on $m_\ell(k, n)$, where the explicit construction of the two families of $k$-graphs can be found in \cite{KOT14,TZ2}:
\begin{equation}\label{equ:f(k)_ell}
\mbox{ By setting } f^\ell(k):=\limsup_{n\rightarrow \infty} \frac{m_\ell(k,n)}{\binom{n-\ell}{k-\ell}}, \mbox{ it holds that } f^\ell(k)\geq \max\left\{\frac12, 1-\left(1-\frac1k\right)^{k-\ell}\right\}.
\end{equation}
For brevity, we abbreviate $m_1(k,n)$ as $m(k,n)$ and $f^1(k)$ as $f(k)$.
Note that the well-known Dirac's theorem implies that $f(2)=1/2$ for graphs.
It is conjectured (see \cite{HPS09,KO09}) that the above lower bound should hold as an equality for all integers $k > \ell \geq 1$.
This conjecture remains open in general, although significant progress has been achieved,
including the confirmation of $f(3) = 5/9$ in \cite{HPS09,Kh13,KOT13}, the confirmation of $f(4) = 1-(3/4)^3$ in \cite{AFH12,K16},
the confirmation of $f(5) = 1 - (4/5)^4$ in \cite{AFH12}, the confirmation of $f^{k-1}(k) = 1/2$ in \cite{KO06,RRS1,RRS2},
as well as the verification of $f^\ell(k) = 1/2$ for $\ell \geq k/2$ in \cite{P08,TZ1} and towards better bounds $\ell\geq 0.4k$ in \cite{Han16,LY22}.
For further discussions, we refer to \cite{KOT14,RR10,Zhao}.

Now we formalize the problem of determining the minimum $\ell$-degree threshold for perfect matchings with high discrepancy in $r$-edge-colored $k$-uniform hypergraphs.
Fix integers $r, k,\ell$ with $r\geq 2$ and $1\leq \ell\leq k-1$.
Let $h_{r}^\ell(k)$ denote the smallest constant $h>0$ satisfying that for any $\epsilon>0$, there exists a constant $\delta=\delta(\epsilon,r,k,\ell)>0$ so that the following holds:
for every $n$-vertex $k$-graph $H$ and every $r$-coloring of $E(H)$, where $n\equiv0\pmod k$ is sufficiently large, if $\delta_\ell(H)\geq (h+\epsilon)\cdot \binom{n-\ell}{k-\ell},$
then there exists a perfect matching of $H$ with at least $\frac{n}{rk}+\delta\cdot n$ edges with the same color.
For brevity, we will write $h_{r}^1(k)$ as $h_{r}(k)$.
It is evident from the definition that
\begin{equation}\label{equ:h>=f}
h_{r}^\ell(k)\geq f^\ell(k) \mbox{ holds for all $r\geq 2$ and $1\leq \ell\leq k-1$.}
\end{equation}
Balogh, Csaba, Jing and Pluh\'ar \cite{BCJP20} first established the minimum degree threshold for perfect matchings with high discrepancy in graphs by showing that $h_2(2)=\frac34$.
Later, this was generalized to the multicolor version in Freschi-Hyde-Lada-Treglown \cite{FHLT} and Gishboliner-Krivelevich-Michaeli \cite{GKM2}, where it was proven that $h_r(2)=\frac{r+1}{2r}$ holds for all $r\geq 2$.
This bound is attained by the following $r$-edge-colored graph $G$: $V(G)$ consists of subsets $V_1,\ldots, V_r$ with $|V_i|=\frac{n}{2r}$ for all $i<r$ and $|V_r|=\frac{(r+1)n}{2r}$;
the edges of $G$ comprises all pairs intersecting with $V_r$, where for each $i\in [r]$ the edges between $V_i$ and $V_r$ are colored by color $i$.
The same problem was also explored in random graphs in \cite{GKM1}.
Recently, Gishboliner, Glock and Sgueglia \cite{GGS} determined the minimum co-degree threshold by deriving from their main result that $h_r^{k-1}(k)=\frac12=f^{k-1}(k)$ holds for all $r\geq 2$ and $k\geq 3$.
Simultaneously and independently, Balogh, Treglown and Z\'arate-Guer\'en \cite{BTZ} demonstrated that $h_r^{\ell}(k)=f^\ell(k)$ holds for all $r\geq 2$ and $2\leq \ell\leq k-1$.
This is remarkable given that generally the exact value of $f^\ell(k)$ remains unknown.
Now we are left with the only remaining cases when $\ell=1$, namely, the minimum vertex degree threshold for perfect matchings with high $r$-color discrepancy in $k$-graphs.
This is summarized in the following problem statement.

\begin{prob}\label{prob}
Determine $h_r(k)$ for all $k\geq 3$ and $r\geq 2$.
\end{prob}

It is worth mentioning that there is a notable difference in the known cases between graphs and hypergraphs.
In graphs, we observe that $h_r(2)>f(2)$ for all $r\geq 2$,
whereas in $k$-uniform hypergraphs for every $k\geq 3$,
all known cases indicate that $h_r^{\ell}(k)=f^\ell(k)$ for all $r\geq 2$, where $2\leq \ell\leq k-1$.
As inquired by Gishboliner, Glock, and Sgueglia \cite{GGS}, it would be interesting to know whether there exists some case in $k$-uniform hypergraphs with $k\geq 3$, where the discrepancy threshold is strictly larger than the corresponding existence threshold (namely, $h_r(k)>f(k)$ in the remaining cases).
In particular, Balogh, Treglown and Z\'arate-Guer\'en \cite{BTZ} considered the discrepancy threshold in 3-graphs.
They observed $h_2(3)\geq \frac{3}{4}$ from the following 3-graph $H$:
The vertex set $V(H)$ consists of two subsets $A$ and $B$, both of size $n/2$,
while $E(H)$ comprises all red-colored edges with two endpoints in $A$ and one endpoint in $B$,
along with all blue-colored edges having one endpoint in $A$ and two endpoints in $B$.
The authors of \cite{BTZ} (see Question 4.2) posed the question of whether $h_2(3)=\frac{3}{4}$ holds true.

\subsection{Our results}
In this paper, we advance and complete the aforementioned line of research by establishing the minimum vertex degree threshold of perfect matchings with high $r$-color discrepancy in $k$-graphs for all $r,k\geq 2$.
To achieve this, we identify this discrepancy threshold using a novel family of $k$-graphs and unveil new phenomena that have not been explored in previous results.

To set the stage for our findings, we initiate our presentation by examining the lower bound with a detailed characterization of the extremal $k$-graphs as follows.
For integers $k\geq 3$ and $r\geq 2$,
let $\mathbb{N}^r_{k-1}$ denote the set consisting of all vectors $\vec{\mathbf{a}} = (a_1,a_2,\ldots,a_r)\in \mathbb{N}^r$ such that $a_1 + a_2+\ldots + a_r = k-1$ and $0\leq a_1\leq a_2\leq \ldots \leq a_r$. (See Figure~\ref{fig} for illustration.)

\begin{myDef}\label{Def:graphs}
For any integer $n$ divisible by $rk$, $\vec{\mathbf{a}} = (a_1,a_2,\ldots,a_r)\in \mathbb{N}^r_{k-1}$ and $\vec{\mathbf{b}}=(b_1,b_2,\ldots,b_r)\in\mathbb{N}_r^{k+1}$ such that $b_1\geq 1$.
Let $H(n,\vec{\mathbf{a}})$ denote the $k$-graph on vertex set $V_1 \cup V_2 \cup \ldots \cup V_r$ with $|V_i|=\frac{ra_i+1}{rk}n$ for each $i\in [r]$ and
consisting of all edges $e$ with $|V(e)\cap V_i|=a_i+1$ and $|V(e)\cap V_j|=a_j$ for every $i\in [r]$ and all $j\in [r]\backslash \{i\}$.
\end{myDef}
For any $k\geq 3$ and $r\geq 2$, we define $\mathcal{H}_{r,k}=\{H(n,\vec{\mathbf{a}})\ |\ \vec{\mathbf{a}}\in \mathbb{N}^r_{k-1}\}$
and
\[g_r(k):=\max_{\vec{\mathbf{a}}\in \mathbb{N}^r_{k-1}} \left\{\lim_{n\rightarrow \infty}\frac{\delta(H(n,\vec{\mathbf{a}}))}{\binom{n-1}{k-1}},\  ~ \bigg\vert ~ n \mbox{ is divisible by } rk \right\}.\footnote{We refer to Lemma~\ref{lem:smallest_deg} for a precise formulation of $g_r(k)$.}\]
It is notable that each $H(n,\vec{\mathbf{a}})$ contains exactly $r$ types of edges:
for every $i\in [r]$, let $E_i$ denote the edge set containing all edges $e$ with $|V(e)\cap V_i|=a_i+1$ and $|V(e)\cap V_j|=a_j$ for each $j\in [r]\backslash \{i\}$.
Let $H^*(n,\vec{\mathbf{a}})$ be the $r$-edge-colored $k$-graph obtained from $H(n,\vec{\mathbf{a}})$ by assigning color $i$ to each edge in $E_i$ for all $i\in [r]$. 
It is observed that every perfect matching in $H^*(n,\vec{\mathbf{a}})$ has exactly $n/rk$ edges with color $i$ for every $i\in [r]$,
thereby indicating that $h_r(k)\geq g_r(k)$.
Combining this with \eqref{equ:h>=f}, we derive the following lower bound for $h_r(k)$.
\begin{prop}\label{prop:main}
For any $k\geq 3$ and $r\geq 2$, it holds that $h_r(k)\geq \max\{f(k), g_r(k)\}.$
\end{prop}

\begin{figure}[ht]
	\centering
		\begin{tikzpicture}[thin][node distance=1cm,on grid]
		\draw[fill = lightgray!30, fill opacity=0.5] (-3.2,-2) arc(180:270:0.5)--(-2.1,-2.5) arc(-90:0:0.5)--(-1.6,1.5) arc(0:90:0.5)--(-2.7,2) arc(90:180:0.5)--cycle;
		\draw (-2.4,-2.8) node{$V_1$};
		\draw[fill = lightgray!30, fill opacity=0.5] (-0.8,-2) arc(180:270:0.5)--(0.3,-2.5) arc(-90:0:0.5)--(0.8,1.5) arc(0:90:0.5)--(-0.3,2) arc(90:180:0.5)--cycle;
		\draw (0,-2.8) node{$V_2$};
		\draw[fill = lightgray!30, fill opacity=0.5] (2.4,-2) arc(180:270:0.5)--(3.5,-2.5) arc(-90:0:0.5)--(4.0,1.5) arc(0:90:0.5)--(2.9,2) arc(90:180:0.5)--cycle;
		\draw (3.2,-2.8) node{$V_r$};
		\draw (1.6,-2.8) node{ $\cdots$};
		
		\draw[fill=blue!80, fill opacity=0.3] (3.2,1.0) arc(-90:90:0.3) -- (-2.4,1.8) arc(90:270:0.5) -- cycle;
		\node[circle,inner sep=0.4mm,draw=black,fill=black](x1)at(-2.4,1.5) {};
		\node[circle,inner sep=0.4mm,draw=black,fill=black](x2)at(-2.4,1.1) {};
		\node[circle,inner sep=0.4mm,draw=black,fill=black](y1)at(0,1.3) {};
		\node[circle,inner sep=0.4mm,draw=black,fill=black](z1)at(3.2,1.3) {};
		
		\draw[fill=red!80, fill opacity=0.3] (3.242,-0.266) arc(-80:80:0.25) -- (0.083,0.49) arc(80:100:0.5) -- (-2.442,0.266) arc(100:260:0.25)-- (-0.083,-0.49) arc(260:280:0.5) --cycle;
		\node[circle,inner sep=0.4mm,draw=black,fill=black](x3)at(-2.2,0) {};
		\node[circle,inner sep=0.4mm,draw=black,fill=black](y2)at(0,0.2) {};
		\node[circle,inner sep=0.4mm,draw=black,fill=black](y3)at(0,-0.2) {};
		\node[circle,inner sep=0.4mm,draw=black,fill=black](z2)at(3.0,0) {};
		
		\draw[fill=green!100, fill opacity=0.3] (-2.4,-1.5) arc(90:270:0.3) -- (3.2,-2.3) arc(-90:90:0.5) -- cycle;
		\node[circle,inner sep=0.4mm,draw=black,fill=black](x4)at(-2.4,-1.8) {};
		\node[circle,inner sep=0.4mm,draw=black,fill=black](y4)at(0,-1.8) {};
		\node[circle,inner sep=0.4mm,draw=black,fill=black](z3)at(3.2,-1.6) {};
		\node[circle,inner sep=0.4mm,draw=black,fill=black](z4)at(3.2,-2.0) {};
		
		\draw (-4,1.3) node{\small $e_1\in E_1$};
		\draw (-4,0) node{\small $e_2\in E_2$};
		\draw (-4,-0.9) node{\small $\vdots$};
		\draw (-4,-1.8) node{\small $e_r\in E_r$};		
	\end{tikzpicture}
	\caption{The $(r+1)$-graph $H(n,\vec{\mathbf{a}})\in  \mathcal{H}_{r,r+1}$ for $\vec{\mathbf{a}}=(1,1,\ldots,1)\in \mathbb{N}^r_{r}$}
\label{fig}
\end{figure}
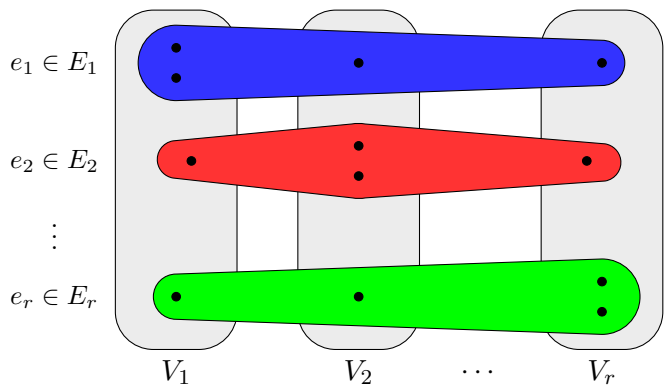

Our main result is stated below, demonstrating that the above lower bound is actually optimal.

\begin{thm}\label{main}
For any integers $k\geq 3$ and $r\geq 2$, there exists a real $\gamma=\gamma(k,r)> 0$ such that the following holds.
For any $\eta>0$ and any $k$-uniform hypergraph $H$ with a sufficiently large number of vertices $n\geq n_0(k,r,\eta)$, where $n\equiv 0\pmod k$,
if
\[\delta(H)>(\max\{f(k), g_r(k)\}+\eta)\cdot \binom{n-1}{k-1}\]
then for any $r$-coloring of the edges of $H$,
there is a perfect matching in $H$ with at least $\frac{n}{rk}(1+\gamma\cdot\eta)$ edges of the same color.
\end{thm}

Now combining Proposition~\ref{prop:main} with Theorem~\ref{main}, we can express the solution to Problem~\ref{prob} as $$h_r(k)= \max\{f(k), g_r(k)\} \mbox{ for all } k\geq 3 \mbox{ and } r\geq 2.$$
For any given values of $k$ and $r$, this involves a finite optimization process,
however determining whether $h_r(k)>f(k)$ is not a straightforward task.
In the subsequent result,
we further identify the unique extremal configurations for all instances, with the exception of finite specific cases, namely when $r=2$ and $6\leq k\leq 16$.

\begin{thm}\label{thm:main2}
For any integers $k\geq 3$ and $r\geq 2$, it holds that $$h_r(k)= \max\{f(k), g_r(k)\},$$
where we have
\begin{itemize}
\item $f(k)>g_r(k)$ if one of the following holds:

(1) $r=2$ and $k\geq 17$;\quad (2) $r=3$ and $k\geq 5$;\quad (3) $r=4$ and $k\geq 4$;\quad (4) $r\geq 5$.
\item $g_r(k)>f(k)$ if one of the following holds:
\begin{itemize}
\item [(1)] $r=2$ and $k=3$, which is attained by the 3-graph $H(n,(1,1))$;
\item [(2)] $r=2$ and $k=4$, which is attained by the 4-graph $H(n,(1,2))$;
\item [(3)] $r=2$ and $k=5$, which is attained by the 5-graph $H(n,(0,4))$;
\item [(4)] $r=3$ and $k=3$, which is attained by the 3-graph $H(n,(0,0,2))$;
\item [(5)] $r=3$ and $k=4$, which is attained by the 4-graph $H(n,(0,0,3))$;
\item [(6)] $r=4$ and $k=3$, which is attained by the 3-graph $H(n,(0,0,0,2))$.
\end{itemize}
\end{itemize}
\end{thm}

The proof of this theorem relies on established results on $f(k)$ and utilizes analytical optimization techniques on $g_r(k)$.
It indicates that apart from a finite number of cases,
the $r$-color discrepancy threshold aligns with the corresponding existence threshold for perfect matchings in $k$-uniform hypergraphs for $k\geq 3$ and $r\geq 2$.
Conversely, notably, there exist instances where the discrepancy threshold exceeds the corresponding existence threshold, thus offering an affirmative answer to the query posed by Gishboliner, Glock, and Sgueglia \cite{GGS}.
Furthermore, observing that the 3-graph $H(n,(1,1))$ coincides with the example provided by Balogh, Treglown, and Z\'arate-Guer\'en \cite{BTZ},
this shows that $h_2(3)=\frac{3}{4}$ indeed holds, confirming the question posted in \cite{BTZ} for the case $k=3$ and $r=2$.

The rest of the paper is organized as follows.
In Section~2, we introduce some necessary notations and preliminary results.
In Section 3, we establish Lemma~\ref{lem:key1}, offering the fundamental structure for the extremal hypergraphs we consider.
In Section~4 and Section~5, we prove Theorem \ref{main} and Theorem~\ref{thm:main2}, respectively.
Finally, in Section~6 we discuss some remarks and questions for further consideration.

\section{Preliminaries}
In this section, we introduce notations and some preliminary results that will be used in the forthcoming proof.
Throughout, let $r\geq 2$ be an integer, $H$ be a $k$-graph and $c$ be an $r$-edge-coloring of $H$ such that $c: E(H)\rightarrow\{C_1,C_2,\ldots,C_r\}$.

We primarily adhere to the terminologies used in \cite{BTZ}.
Given $x\in V(H)$ and $A\subseteq V(H)$, we write $xA$ or $Ax$ to denote the set $\{x\}\cup A$.
By $H-A$, we mean the $k$-graph obtained from $H$ by deleting all vertices in $A$.
An edge of $H$ is called a $C_i$-edge if it is colored by $C_i$ in $c$.
For a subgraph $F$ of $H$, the {\it color profile} $\cp(F)$ of $F$ denotes the $r$-tuple $(x_1,x_2,\ldots, x_r)$,
where $x_i$ denotes the number of $C_i$-edges in $F$ for each $i\in [r]$.
The following definitions are crucial in the coming proofs.

\begin{myDef}\label{Def:2vtx}
Let $\{u, v\}\in \binom{V(H)}{2}$ and $T\in N_H(u)\cap N_H(v)$.
We say $uTv$ is
\begin{itemize}
  \item $\bfS$, if $c(T\cup \{u\}) = c(T\cup \{v\})$, or
  \item $\bfi\bfj$, if $c(T\cup \{u\}) = C_i$ and $c(T\cup \{v\}) = C_j$ where $i\neq j$.
\end{itemize}
\end{myDef}

\begin{myDef}\label{Def:type}
Let $\{u, v\}\in \binom{V(H)}{2}$. We say that an ordered pair $(u,v)$ is
\begin{itemize}
\item type $\bfS$ if there are at least $k^2\binom{n-2}{k-2}$ sets $T\in N_H(u)\cap N_H(v)$ such that $uTv$ is $\bfS$, or
\item type $\bfi\bfj$ if there are at least $k^2\binom{n-2}{k-2}$ sets $T\in N_H(u)\cap N_H(v)$ such that $uTv$ is $\bfi\bfj$.
\end{itemize}
\end{myDef}

We point out that $(u,v)$ is type $\bfi\bfj$ if and only if $(v,u)$ is type $\bfj\bfi$, and it is possible that $(u,v)$ has multiple types.
In the coming proofs, we will frequently use the following simply observations:
If $(u,v)$ is type $\bfS$ (or $\bfi\bfj$),
then for any subset $A$ of size less than $k^2$, there exists a set $T$ in $N_H(u)\cap N_H(v)$ such that $T\cap A=\emptyset$ and $uTv$ is $\bfS$ (or $\bfi\bfj$);
in particular, one can find at least $k$ disjoint sets $T_i$ in $N_H(u)\cap N_H(v)$ such that each $uT_iv$ is $\bfS$ (or $\bfi\bfj$).

A subgraph $F$ of $H$ is called {\it good} (with respect to a given coloring $c$), if $F$ contains two perfect matchings that have different color profiles.
As we will see later, it becomes evident that good subgraphs are essential in the construction of color-biased perfect matchings of $H$.

In what follows, we introduce three kinds of good subgraphs (referred to as {\it gadgets}) that will be heavily used in the proofs.

\begin{gad}\label{gad1}
Let $u, v\in V(H)$ be distinct vertices.
A good gadget $G$ of the first kind is a subgraph of $H$ on $2k$ vertices such that
\begin{itemize}
  \item $G:=H[\{u,v\}\cup T_1\cup T_2]$, where $T_1,T_2\in N_H(u)\cap N_H(v)$ are disjoint, and
  \item $\{T_1u,T_2v\}$ and $\{T_1v,T_2u\}$ are two perfect matchings of $G$ with different color profiles.
\end{itemize}
\end{gad}

\begin{gad}\label{gad2}
Let $\{u_1,u_2,u_3\}\in \binom{V(H)}{3}$. A good gadget $G$ of the second kind is a subgraph of $H$ on $3k$ vertices such that
\begin{itemize}
  \item  $G:=H[\{u_1,u_2,u_3\}\cup T_1\cup T_2\cup T_3]$, where $T_1\in N_H(u_2)\cap N_H(u_3)$, $T_2\in N_H(u_1)\cap N_H(u_3)$ and $T_3\in N_H(u_1)\cap N_H(u_2)$ are pairwise disjoint, and
  \item $\{T_1u_2,T_2u_3,T_3u_1\}$ and $\{T_1u_3,T_2u_1,T_3u_2\}$ are two perfect matchings of $G$ with two different color profiles.
\end{itemize}
\end{gad}

\begin{gad}\label{gad3}
Let $e,f$ be two edges of $H$ such that $V(e)\backslash V(f)=\{u_1,\ldots,u_\ell\}$ and $V(f)\backslash V(e)=\{v_1,\ldots,v_\ell\}$ for some $1\leq \ell\leq k$.
A good gadget $G$ of the third kind is a subgraph of $H$ on $(\ell+1)k$ vertices such that
\begin{itemize}
  \item  $G:=H[V(e)\cup V(f)\cup T_1\cup \ldots \cup T_\ell]$, where $T_i\in N_H(u_i)\cap N_H(v_i)$ for all $i\in [\ell]$ are disjoint, and
  \item $\{u_1T_1,\ldots,u_\ell T_\ell,f\}$ and $\{v_1T_1,\ldots,v_\ell T_\ell,e\}$ are two perfect matchings of $G$ with different color profiles.
\end{itemize}
\end{gad}

Finally, we need the following classic result of Hilton and Milner \cite{HM}.
We say that two families $\mathcal{A},\mathcal{B}$ of sets are {\it cross-intersecting}, if $A\cap B\neq \emptyset$ for every $A\in \mathcal{A}$ and every $B\in \mathcal{B}$.

\begin{thm}[Hilton and Milner \cite{HM}]\label{cross}
	Suppose that  $m> 2\ell>0$ and $\mathcal{A},\mathcal{B}\subseteq \binom{[m]}{\ell}$ are two non-empty and cross-intersecting families. Then we have
	\begin{equation*}
		|\mathcal{A}|+|\mathcal{B}|\leq\binom{m}{\ell}-\binom{m-\ell}{\ell}+1.
	\end{equation*}
\end{thm}

\section{Extremal structure and minimum degree}
In this section, we prove our key technical lemma -- Lemma~\ref{lem:key1},
which forms the structural foundation of the main theorem.

Before introducing this lemma,
we define $k$-graphs that extend Definition~\ref{Def:graphs} and highlight a crucial observation.
Let $n, k\geq 3$ and $r\geq 2$ be integers such that $n$ is dividable by $rk$.
For any $\vec{\mathbf{a}}=(a_1,\ldots,a_r)\in \mathbb{N}^r_{k-1}$ and any $r$-partition $\vec{V}=(V_1,\ldots, V_r)$ of $n$ vertices,
let $H(\vec{V},\vec{\mathbf{a}})$ denote the $k$-graph with vertex set $V_1\cup V_2\cup ...\cup V_r$ consisting of all edges $e$ with $|V(e)\cap V_i|=a_i+1$ and $|V(e)\cap V_j|=a_j$ for all $1\leq i\neq j\leq r$.
It is not hard to see that $H(\vec{V},\vec{\mathbf{a}})$ contains a perfect matching if and only if $|V_i|\geq \frac{a_i}{k}n$ holds for every $i\in [r]$.
Let $H^*(\vec{V},\vec{\mathbf{a}})$ be the $r$-edge-colored $k$-graph obtained from $H(\vec{V},\vec{\mathbf{a}})$ by coloring each edge $e$, satisfying $|V(e)\cap V_i|=a_i+1$, with the color $C_i$ for every $i\in [r]$.
For a subgraph $F$, we write $C_i(F)$ as the number of $C_i$-edges contained in $F$.
Then one can observe that for every perfect matching $\mathcal{M}$ in $H^*(\vec{V},\vec{\mathbf{a}})$ and for every $i\in [r]$,
\begin{equation}\label{equ:obs2}
\mbox{it holds that $|V_i|=a_i\frac{n}{k}+C_i(\mathcal{M})$.}
\end{equation}

There is still another type of $k$-graphs such that there does not exist any kind of good gadgets for some $r$-coloring.
Indeed, it has similar structure as $H(n,\vec{\mathbf{a}})$.
Let $n, k\geq 3$ and $r\geq 2$ be integers such that $n$ is dividable by $rk$.
For any $\vec{\mathbf{b}}\in\mathbb{N}_{k+1}^r$ such that $b_1\geq 1$,
let $\tilde{H}(n,\vec{\mathbf{b}})$ denote the the $k$-graph on vertex set $V_1 \cup V_2 \cup \ldots \cup V_r$ with $|V_i|=\frac{rb_i-1}{rk}n$ for each $i\in [r]$ and
consisting of all edges $e$ with $|V(e)\cap V_i|=b_i-1$ and $|V(e)\cap V_j|=b_j$ for every $i\in [r]$ and all $j\in [r]\backslash \{i\}$.
Let $\tilde{H}^*(n,\vec{\mathbf{b}})$ be the $r$-edge-colored $k$-graph obtained from $H(n,\vec{\mathbf{b}})$ by assigning color $i$ to each edge in $E_i$ for all $i\in [r]$, where $E_i$ consists of all edges $e$ in $H(n,\vec{\mathbf{b}})$ such that $|V(e)\cap V_i|=b_i-1$ and $|V(e)\cap V_j|= b_j$ for each $j\in [r]\backslash\{i\}$.
It is easy to check that every perfect matching of $\tilde{H}^*(n,\vec{\mathbf{b}})$ has exactly $n/rk$ edges with color $i$ for every $i\in[r]$.

Under the same setting of $n,k$ and $r$, for any $\vec{\mathbf{b}}=(b_1,\ldots, b_r)\in\mathbb{N}_{k+1}^r$ satisfying $b_1\geq 1$ and any $r$-partition $\vec{V}=(V_1,\ldots, V_r)$ of $n$ vertices, 
let $\tilde{H}(\vec{V},\vec{\mathbf{b}})$ denote the $k$-graph with vertex set $V_1\cup V_2\cup ...\cup V_r$ consisting of all edges $e$ with $|V(e)\cap V_i|=b_i-1$ and $|V(e)\cap V_j|=b_j$ for all $1\leq i\neq j\leq r$.
Let $\tilde{H}^*(\vec{V},\vec{\mathbf{b}})$ be the $r$-edge-colored $k$-graph obtained from $\tilde{H}(\vec{V},\vec{\mathbf{b}})$ by coloring each edge $e$, satisfying $|V(e)\cap V_i|=b_i-1$, with the color $C_i$ for every $i\in [r]$.
Then one can observe that for every perfect matching $\mathcal{M}$ in $\tilde{H}^*(\vec{V},\vec{\mathbf{b}})$ and for every $i\in [r]$,
\begin{equation}\label{equ:obs2}
\mbox{it holds that $|V_i|=b_i\frac{n}{k}-C_i(\mathcal{M})$.}
\end{equation}

Our key lemma reads as follows.
It establishes the necessity of the presence of $H^*(\vec{V},\vec{\mathbf{a}})$ or $\tilde{H}^*(\vec{V},\vec{\mathbf{b}})$ for some specific $\vec{V}$ and $\vec{\mathbf{a}}\in\mathbb{N}^r_{k-1}$ or $\vec{\mathbf{b}}\in\mathbb{N}^r_{k+1}$ satisfying $b_1\geq 1$
in determining the extremal structure of $k$-graphs.

\begin{lem}\label{lem:key1}
Let $n,k, r$ be positive integers with $n\geq 2k^2$ and $r\geq 2$.
Let $H$ be an $r$-edge-colored $n$-vertex $k$-graph with the edge-coloring $c: E(H)\rightarrow\{C_1,C_2,\ldots,C_r\}$ such that $H$ has at least one $C_i$-edge for each $i\in [r]$.
Suppose the following conditions hold:
\begin{itemize}
\item $\delta(H)>\frac{1}{2}\binom{n-1}{k-1}+\frac{k^2+1}{2}\binom{n-2}{k-2}$, and
\item $H$ does not contain any good gadgets of any kind.
\end{itemize}
Then by possibly renaming the colors, there exists an $r$-partition $\vec{V}$ of $V(H)$ and a vector $\vec{\mathbf{a}}\in\mathbb{N}^r_{k-1}$ or a vector $\vec{\mathbf{b}}\in \mathbb{N}^r_{k+1}$ satisfying $b_1\geq 1$ such that $H$ is an $r$-edge-colored subgraph of $H^*(\vec{V},\vec{\mathbf{a}})$ or $\tilde{H}^*(\vec{V},\vec{\mathbf{b}})$.
\end{lem}

In the remainder of this section, we start by establishing several lemmas that describe the local structures of $k$-graphs with high minimum degrees but without any good gadgets.
Following that, we present the proof of Lemma~\ref{lem:key1} in the next subsection.
Finally, we conclude this section with a lemma that describes the minimum vertex degree of the extremal $k$-graphs under consideration.

\subsection{Local structures}\label{subsec:local}
Now, we proceed to prove a series of three lemmas, with each lemma building upon the previous one.
These lemmas are aimed at describing the local structure of hypergraphs using the terminologies provided in Definitions~\ref{Def:2vtx} and \ref{Def:type}.

\begin{lem}\label{lem1}
Let $n,k, r$ be positive integers with $n>2k$ and $r\geq 2$.
Let $H$ be an $n$-vertex $k$-graph and let $c: E(H)\rightarrow\{C_1,C_2,\ldots,C_r\}$.
Assume that $H$ does not contain any good gadgets of any kind.
Let $\{u,v\}\in \binom{V(H)}{2}$ satisfy
\begin{equation}\label{equ:NuCAPNv}
|N_H(u)\cap N_H(v)|> \binom{n-2}{k-1}-\binom{(n-2)-(k-1)}{k-1}+1.
\end{equation}
Then there exists $\mathbf{U}=\bfS$ or $\bfi\bfj$ for some $i\neq j$ such that
$uTv$ is $\mathbf{U}$ for all $T\in N_H(u)\cap N_H(v)$.
\end{lem}

\begin{proof}
We first claim that if there exists $T_1\in N_H(u)\cap N_H(v)$ such that $uT_1v$ is $\bfS$,
then for all $T\in N_H(u)\cap N_H(v)$, $uTv$ is $\bfS$.
Suppose this is not the case.
Then $\mathcal{A}:=\{T\in N_H(u)\cap N_H(v): uTv \mbox{ is } \bfS\}$ and $\mathcal{B}:=\{T\in N_H(u)\cap N_H(v): uTv \mbox{ is not } \bfS\}$ both are non-empty families contained in $\binom{V(H)\backslash \{u,v\}}{k-1}$.
So $|\mathcal{A}|+|\mathcal{B}|=|N_H(u)\cap N_H(v)|$ satisfies \eqref{equ:NuCAPNv}.
By Lemma~\ref{cross}, there exist two disjoint $(k-1)$-subsets $R_1, R_2\in N_H(u)\cap N_H(v)$ such that $uR_1v$ is $\bfS$ and $uR_2v$ is not $\bfS$.
So $uR_2v$ is $\bfi\bfj$ for some $1\leq i\neq j\leq r$.
Then $\{uR_1,vR_2\}$ and $\{vR_1,uR_2\}$ are two perfect matchings of $H[R_1\cup R_2\cup \{u,v\}]$ with different color profiles.
This gives a good gadget of the first kind, a contradiction.

By the this claim, we may assume that for every $T\in N_H(u)\cap N_H(v)$, $uTv$ is not $\bfS$.
We need to show they must belong to the same kind, say $\bfi\bfj$ for some $i\neq j$.
Suppose for a contradiction that this is not the case.
Then using Lemma~\ref{cross} (and the argument in the previous paragraph) again,
we see that there exist two disjoint $(k-1)$-subsets $R_1, R_2\in N_H(u)\cap N_H(v)$ such that $uR_1v$ is $\bfi\bfj$ and $uR_2v$ is not $\bfi\bfj$.
We may assume that $uR_2v$ is $\bfl\bfm$, where either $i\neq \ell$ or $j\neq m$.
It is evident to see that in either case, $\{uR_1,vR_2\}$ and $\{vR_1,uR_2\}$ are two perfect matchings of $H[T_1\cup T_2\cup \{u,v\}]$ with different color profiles, again a contradiction.
This completes the proof.
\end{proof}

The subsequent lemma shows that under some mild assumptions on hypergraphs, each ordered pair of vertices possesses a unique type.

\begin{lem}\label{lem:2vtx}
Let $n, k, r$ be integers with $n\geq 2k^2$ and $r\geq 2$.
Let $H$ be an $n$-vertex $k$-graph and let $c: E(H)\rightarrow\{C_1,C_2,\ldots,C_r\}$.
Assume that $H$ does not contain any good gadgets of any kind and
\begin{equation}\label{equ:min-deg}
\delta(H)>\frac{1}{2}\binom{n-1}{k-1}+\frac{k^2+1}{2}\binom{n-2}{k-2}.
\end{equation}
Then the following hold:
\begin{itemize}
\item[(A).] Every ordered pair $(u,v)$ of $V(H)$ has a unique type, which is either type $\bfS$ or type $\bfi\bfj$ for some $1\leq i\neq j\leq r$;
\item[(B).] If $c$ is not monochromatic, then there exists $\{u,v\}\in \binom{V(H)}{2}$ such that $(u,v)$ is not type $\bfS$.
\end{itemize}
\end{lem}

\begin{proof}
For (A), consider any $\{u,v\}\in \binom{V(H)}{2}$.
Since $N_H(u)\cap N_H(v)$ is contained in $\binom{V(H)\backslash \{u,v\}}{k-1}$, using \eqref{equ:min-deg} we can derive that $|N_H(u)\cap N_H(v)|$ is at least
\begin{align*}
&\left|N_H(u)-\{v\}\right|+\left|N_H(v)-\{u\}\right|-\binom{n-2}{k-1}\geq 2\left(\delta(H)-\binom{n-2}{k-2}\right)-\binom{n-2}{k-1}\\
> & k^2\binom{n-2}{k-2}>\sum_{i=1}^{k-1}\binom{n-i-2}{k-2}+1= \binom{n-2}{k-1}-\binom{(n-2)-(k-1)}{k-1}+1.
\end{align*}
That is, $|N_H(u)\cap N_H(v)|> k^2\binom{n-2}{k-2}$ and it satisfies \eqref{equ:NuCAPNv}.
By Lemma~\ref{lem1}, there exists $\mathbf{U}=\bfS$ or $\bfi\bfj$ for some $i\neq j$ such that
$uTv$ is $\mathbf{U}$ for all $T\in N_H(u)\cap N_H(v)$.
Hence, $(u,v)$ has a unique type, which is either type $\bfS$ or type $\bfi\bfj$.

To show (B), suppose for a contradiction that any pair $(u,v)$ of $V(H)$ is type $\bfS$.
Assume that $H$ contains a $C_1$-edge.
Let $\mathcal{A}$ be the family consisting of all $C_1$-edges in $H$ and $\mathcal{B}=E(H)\backslash \mathcal{A}$.
Since $c$ is not monochromatic, $\mathcal{A}$ and $\mathcal{B}$ are non-empty.
By \eqref{equ:min-deg} and since $n\geq 2k^2$, it is easy to see that
$$|\mathcal{A}|+|\mathcal{B}|=e(H)>\left(\frac12\binom{n-1}{k-1}+1\right)\frac{n}{k}\geq k\binom{n-1}{k-1}+2k\geq \binom{n}{k}-\binom{n-k}{k}+2.$$
By Lemma~\ref{cross}, there exist two disjoint edges $e,f$ in $H$ such that $e$ is a $C_1$-edge and $f$ is not.
Write $V(e)=\{u_1,\ldots,u_k\}$ and $V(f)=\{v_1,\ldots,v_k\}$.
Since every $(u_i,v_i)$ is type $\bfS$ for $i\in [k]$,
by the observations after Definition~\ref{Def:type},
there exist $T_i\in N_H(u_i)\cap N_H(v_i)$ for $i\in [k]$ which are pairwise disjoint.
Then $\{u_1T_1,\ldots,u_kT_k,f\}$ and $\{v_1T_1,\ldots,v_kT_k,e\}$ are two perfect matchings of $H_0=H[V(e)\cup V(f)\cup (\bigcup_{i=1}^k T_i)]$ with different color profiles,
which shows that $H_0$ is a good gadget of the third kind, a contradiction.
This finishes the proof.
\end{proof}

The final lemma in this subsection characterizes all possible pairwise type situations between any three vertices.

\begin{lem}\label{lem:3vtx}
Let $n,k, r$ be positive integers with $n\geq 2k^2$ and $r\geq 2$.
Let $H$ be an $n$-vertex $k$-graph and let $c: E(H)\rightarrow\{C_1,C_2,\ldots,C_r\}$.
Assume that $H$ satisfies \eqref{equ:min-deg} and does not contain any good gadgets of any kind.
Then the following hold for any $\{u_1,u_2,u_3\}\in \binom{V(H)}{3}$.
\begin{itemize}
\item[(A).] If $(u_1,u_2)$ is type $\bfS$ and $(u_1,u_3)$ is type $\bfi\bfj$, then $(u_2,u_3)$ is also type $\bfi\bfj$;
\item[(B).] If every $(u_i,u_j)$ is not type $\bfS$ for $1\leq i<j\leq 3$, then there exist distinct $\alpha, \beta, \gamma\in [r]$
such that one of the following happens: 
\begin{itemize}
    \item[(a)] $(u_1,u_2)$ is type ${\bf C_\alpha C_\beta}$, $(u_2,u_3)$ is type ${\bf C_\beta C_\gamma}$, and $(u_1,u_3)$ is type ${\bf C_\alpha C_\gamma}$;
    \item[(b)] $(u_1,u_2)$ is type ${\bf C_\beta C_\alpha}$, $(u_2,u_3)$ is type ${\bf C_\gamma C_\beta}$, and $(u_1,u_3)$ is type ${\bf C_\gamma C_\alpha}$;
\end{itemize}
\end{itemize}
\end{lem}

\begin{proof}
For (A), suppose that $(u_2,u_3)$ is not type $\bfi\bfj$.
By Lemma~\ref{lem:2vtx}, $(u_2,u_3)$ is type $\bfS$ or $\bfl\bfm$ where $\ell\neq i$ or $m\neq j$.
Then there exist three disjoint sets $T_3\in N_{H}(u_1)\cap N_H(u_2)$, $T_2\in N_{H}(u_1)\cap N_H(u_3)$, and $T_1\in N_{H}(u_2)\cap N_H(u_3)$ such that
$$\mbox{$u_1T_3u_2$ is $\bfS$, ~ $u_1T_2u_3$ is $\bfi\bfj$, ~ and $u_2T_1u_3$ is either $\bfS$ or $\bfl\bfm$.}$$
We claim that in either case, $\{u_1T_3,u_2T_1,u_3T_2\}$ and $\{u_2T_3,u_3T_1,u_1T_2\}$ are two perfect matchings of $H_1=H[\{u_1,u_2,u_3\}\cup T_1\cup T_2\cup T_3]$ with two different color profiles.
This is easy to see when $u_2T_1u_3$ is $\bfS$.
Now consider when $u_2T_1u_3$ is $\bfl\bfm$.
By symmetry between $\ell$ and $m$, we may assume that $\ell\neq i$.
In this case, we have $c(u_1T_3)=c(u_2T_3)$, and $c(u_2T_1)=\ell\notin \{m,i\}=\{c(u_3T_1),c(u_1T_2)\}$,
confirming the claim.
Hence $H_1$ is a good gadget of the second kind, a contradiction.


Finally we consider (B).
Suppose that $(u_1,u_2)$ is type $\bfi\bfj$, $(u_1,u_3)$ is type $\bfl\bfm$ and $(u_2,u_3)$ is type ${\bf C_gC_h}$.
Then we may choose three disjoint sets $T_3\in N_H(u_1)\cap N_H(u_2)$, $T_2\in N_H(u_1)\cap N_H(u_3)$ and $T_1\in N_H(u_2)\cap N_H(u_3)$
such that
\[\mbox{$u_1T_3u_2$ is $\bfi\bfj$, ~ $u_1T_2u_3$ is $\bfl\bfm$,~ and $u_2T_1u_3$ is ${\bf C_gC_h}$.}\]
Consider the multi-set $L=\{c(u_1T_3),c(u_1T_2),c(u_2T_1),c(u_2T_3),c(u_3T_1),c(u_3T_2)\}$.
If there exists one color appearing in $L$ once or three times,
then clearly, $\{u_1T_3,u_2T_1,u_3T_2\}$ and $\{u_2T_3,u_3T_1,u_1T_2\}$ are two perfect matchings of $H[\{u_1,u_2,u_3\}\cup T_1\cup T_2\cup T_3]]$ with different color profiles, a contradiction.
If a color appears at least four times, then one of $(u_1, u_2)$, $(u_1, u_3)$ and $(u_2, u_3)$ must be type $\mathbf{S}$, a contradiction.
So every color must appear exactly two times in $L$ and thus $L$ contains exactly three distinct colors. 
Furthermore, for every color appears in $L$, it must appears in $\{c(u_1T_3), c(u_2T_1), c(u_3T_2)\}$ and $\{c(u_2T_3), c(u_3T_1), c(u_1T_2)\}$ once respectively.
By definition, we have $(c(u_1T_3), c(u_2T_1), c(u_3T_2))=(\mathbf{C_i},\mathbf{C_g},\mathbf{C_m})$ and $(c(u_2T_3), c(u_3T_1), c(u_1T_2)) =(\mathbf{C_j},\mathbf{C_h},\mathbf{C_\ell})$, then
\begin{equation}\label{equ:triangle}
\{\mathbf{C_i},\mathbf{C_g},\mathbf{C_m}\}=\{\mathbf{C_j},\mathbf{C_h},\mathbf{C_\ell}\},
\end{equation}
implying that $\mathbf{C_i}=\mathbf{C_j}$ or $\mathbf{C_i}=\mathbf{C_h}$ or $\mathbf{C_i}=\mathbf{C_\ell}$.
Since $(u_1, u_2)$ is not type $\mathbf{S}$, then $\mathbf{C_i}\neq \mathbf{C_j}$.
So either $\mathbf{C_i} = \mathbf{C_\ell}$ or $\mathbf{C_i} = \mathbf{C_h}$.
If $\mathbf{C_i} = \mathbf{C_\ell}$, then suppose $(u_1, u_2)$ is type $\mathbf{C_\alpha} \mathbf{C_\beta}$, $(u_1,u_3)$ is type $\mathbf{C_\alpha} \mathbf{C_\gamma}$.
Since $(u_1, u_3),\ (u_2, u_3)$ are not type $\mathbf{S}$, we have $\alpha,\beta,\gamma$ are distinct.
By \eqref{equ:triangle}, we have $(u_2, u_3)$ is type $\mathbf{C_\beta} \mathbf{C_\gamma}$, which means item $(a)$ happens.
Otherwise we have $\mathbf{C_i} = \mathbf{C_h}$. 
In this case, suppose that $(u_1, u_2)$ is type $\mathbf{C_\beta} \mathbf{C_\alpha}$, $(u_2,u_3)$ is type $\mathbf{C_\gamma} \mathbf{C_\beta}$. 
Since $(u_1, u_3),\ (u_2, u_3)$ are not type $\mathbf{S}$, we have $\alpha,\beta,\gamma$ are distinct.
By \eqref{equ:triangle}, we have $(u_1, u_3)$ is type $\mathbf{C_\gamma} \mathbf{C_\alpha}$, which means item $(b)$ happens.
This completes the proof of Lemma \ref{lem:key}.
\end{proof}

\subsection{Proof of Lemma~\ref{lem:key1}}
In this subsection, we prove Lemma~\ref{lem:key1}.
It suffices to establish the following technically equivalent lemma.
Under reasonable assumptions, it not only yields a useful partition of the vertex set of $k$-graphs but also provides crucial information regarding the precise locations and colors of all edges.

\begin{lem}\label{lem:key}
Let $n,k, r$ be positive integers with $n\geq 2k^2$ and $r\geq 2$.
Let $H$ be an $n$-vertex $k$-graph and let $c: E(H)\rightarrow\{C_1,C_2,\ldots,C_r\}$ such that
$H$ has at least one $C_i$-edge for all $i\in [r]$.
Assume that $H$ does not contain any good gadgets of any kind and
\begin{equation*}
\delta(H)>\frac{1}{2}\binom{n-1}{k-1}+\frac{k^2+1}{2}\binom{n-2}{k-2}.
\end{equation*}
Then by possibly renaming the colors, there exist a partition $V(H)=V_1\cup V_2\cup \ldots \cup V_r$  such that the following hold.
\begin{itemize}
\item[(A).] For all $i\in [r]$, every $(u,v)\in V_i\times V_i$ is type $\bfS$;
\item[(B).] One of the following holds:
\begin{itemize}
\item [(B1)] For all $1\leq i<j\leq r$, every $(u,v)\in V_i\times V_j$ is type ${\bf C_{i} C_{j}}$;
\item [(B2)] For all $1\leq i<j\leq r$, every $(u,v)\in V_i\times V_j$ is type ${\bf C_{j} C_{i}}$; and
\end{itemize}
\item[(C).] 
If item (B1) holds, then
there exists non-negative integers $a_1,a_2,\ldots, a_r$ with $\sum_{i=1}^r a_i=k-1$ such that
for all $i\in [r]$, every $C_i$-edge $e$ of $H$ satisfies that $|V(e)\cap V_i|=a_i+1$ and $|V(e)\cap V_j|=a_j$ for every $j\in [r]\backslash \{i\}$.\\
If item (B2) holds, then there exists positive integers $b_1,b_2,\ldots, b_r$ with $\sum_{i=1}^r b_i=k+1$ such that for all $i\in [r]$, every $C_i$-edge $e$ of $H$ satisfies that $|V(e)\cap V_i|=b_i-1$ and $|V(e)\cap V_j|=b_j$ for every $j\in [r]\backslash \{i\}$.
\end{itemize}
\end{lem}

\begin{proof}
Consider $H$ and $c$ as given in the lemma.
It is convenient for us to define an auxiliary graph $G_H$, where it has the vertex set $V(H)$ and the edge set
$$E(G_H)=\{uv\ |\ (u,v) \mbox{ is not type } \bfS\}.$$
By Lemma \ref{lem:2vtx} (B), we see that $E(G_H)\neq \emptyset.$
We first claim that $G_H$ is connected.
Suppose not. Then there exist two connected components say $D_1, D_2$ of $G_H$ such that $D_1$ contains an edge say $u_1u_2$.
Let $v$ be a vertex in $D_2$.
By definition, $(u_1,u_2)$ is not type $\bfS$ and $(u_i,v)$ is type $\bfS$ for each $i=1,2$.
However, this contradicts Lemma~\ref{lem:3vtx} (A),
completing the proof of the claim.

For any $\{v_1,v_2,v_3\}\in \binom{V(H)}{3}$, if $(v_1,v_2)$ and $(v_1,v_3)$ are type $\bfS$, then by Lemma~\ref{lem:3vtx} (A) again,
we see that $(v_2,v_3)$ must be type $\bfS$ as well.
By repeatedly utilizing this fact, it can be deduced that the complement of $G_H$ comprises vertex-disjoint cliques.
In other words, $G_H$ can be represented as a complete multipartite graph, say with parts $V_1, \ldots, V_s$.

We first consider $s\geq 3$. 
Suppose that there are fixed vertices $u_i\in V_i$ for each $i\in [s]$.
By applying Lemma~\ref{lem:3vtx} (B), by possibly renaming the colors, one of the following happens:
\begin{itemize}
\item [(1)] $(u_1,u_2)$ is type $\mathbf{C_1}\mathbf{C_2}$, $(u_1,u_3)$ is type $\mathbf{C_1}\mathbf{C_3}$;
\item [(2)]  $(u_1,u_2)$ is type $\mathbf{C_2}\mathbf{C_1}$, $(u_1,u_3)$ is type $\mathbf{C_3}\mathbf{C_1}$;
\end{itemize}
Suppose item $(1)$ happens. 
In this case, for $4\leq j\leq s$, since $(u_1, u_2),\ (u_1, u_j)$ and $(u_2,u_j)$ are not type $\mathbf{S}$.
Then apply Lemma~\ref{lem:3vtx} (B) again, $(u_1, u_j)$ is type $\mathbf{C_1}\mathbf{C_\alpha}$ or $\mathbf{C_\beta}\mathbf{C_2}$ for some $\alpha, \beta \in \{3,4,\ldots,r\}$.
But if $(u_1, u_j)$ is type $\mathbf{C_\beta}\mathbf{C_2}$, then $(u_1, u_3)$, $(u_1,u_j)$ and $(u_3,u_j)$ are not type $\mathbf{S}$ and do not satisfy any case of Lemma~\ref{lem:3vtx} (B), a contradiction.
So $(u_1, u_j)$ is type $\mathbf{C_1}\mathbf{C_\alpha}$.

Now let $2\leq t\leq s$ be the largest integer such that we can rename the colors letting $(u_1,u_j)$ be type $\mathbf{C_1}\mathbf{C_j}$ for $2\leq j\leq t$.
Since item $(1)$ happens, we have $t\geq 3$.
If $t < s$, then after possibly renaming the colors, we have $(u_1, u_j)$ is type $\mathbf{C_1}\mathbf{C_j}$ for $2\leq j\leq t$.
Suppose that $(u_1,u_{t+1})$ is type $\mathbf{C_1}\mathbf{C_\alpha}$.
By Lemma \ref{lem:3vtx} (B), since $(u_1,u_j)$, $(u_1,u_{t+1})$ and $(u_j,u_{t+1})$ are not type $\mathbf{S}$ for $2\leq j\leq t$, we have $\alpha\neq j$.
So $\alpha \notin\{1,2,\ldots,t\}$, we can rename color $\alpha$ to $t+1$ so that $(u_1,u_{t+1})$ is type $\mathbf{C_1}\mathbf{C_{t+1}}$, a contradiction.
Hence $t = s$, implying that by possibly renaming the colors, we have for every $2\leq j\leq s$, $(u_1,u_j)$ must be type $\mathbf{C_1}\mathbf{C_j}$.
This implies that for every $1\leq i < j \leq s$, $(u_i,u_j)$ is type $\mathbf{C_i}\mathbf{C_j}$.
In particular, this implies that $s\leq r$.
Moreover, we can utilize Lemma~\ref{lem:3vtx} (B) to derive that
\begin{equation}\label{equ:s-parts}
\mbox{every pair $(v_i,v_j)\in V_i\times V_j$ is type $\bfi\bfj$ for all $1\leq i\neq j\leq s$.}
\end{equation}
If item $(2)$ happens, we can derive the following analogously:
\begin{equation}\label{equ:s-parts2}
\mbox{every pair $(v_i,v_j)\in V_i\times V_j$ is type $\bfj\bfi$ for all $1\leq i\neq j\leq s$.}
\end{equation}
If $s=2$, it is no difference between items \eqref{equ:s-parts} and \eqref{equ:s-parts2} and we can utilize Lemma~\ref{lem:3vtx} (B) to prove the same statement.

To complete the proof, it remains to show $s=r$ and the item (C).
For an edge $e\in E(H)$, let $\vec{\#}(e)=(t_1,\ldots, t_s,0,\ldots,0)$ be the $r$-tuple, where each $t_j=|V(e)\cap V_j|$ denotes the number of vertices of $V(e)$ contained in the part $V_j$ for $j\in [s]$.
Note that $\sum_{j=1}^s t_j=k$.
If $e$ is a $C_i$-edge for some $i\in [r]$, we define $\vec{1}(e)$ as the $r$-tuple with its $i^{th}$ entry being $1$ and all other entries being $0$.

Suppose that \eqref{equ:s-parts} happens, we claim that $\vec{\#}(e)-\vec{1}(e)$ is an invariant, denoted as $\vec{a}=(a_1,\ldots,a_r)$, for all edges $e$ in $H$.
It suffices to show that for any two edges $e,f$ in $H$,
it holds that $\vec{\#}(e)+\vec{1}(f)=\vec{\#}(f)+\vec{1}(e)$.
For any given $e,f\in E(H)$, there exists $\ell\in [k]$ such that
$V(e)\setminus V(f)=\{u_1,\ldots,u_\ell\}$ and $V(f)\setminus V(e)=\{v_1,\ldots,v_\ell\}$.
Using similar arguments as before (i.e., Lemma~\ref{lem:2vtx} (A)),
there exist $T_i\in N_H(u_i)\cap N_H(v_i)$ for $1\leq i\leq \ell$ which are pairwise disjoint.
Then $M_1=\{u_1T_1,\ldots,u_\ell T_\ell,f\}$ and $M_2=\{v_1T_1,\ldots,v_\ell T_\ell,e\}$ are two perfect matchings of $H_0=H[V(e)\cup V(f)\cup (\bigcup_{i=1}^\ell T_i)]$.
Let $\vec{c}_i$ be the color profile of the matching $M_i$ for $i=1,2.$
We may assume that $\vec{c}_1=\vec{c}_2$ (as otherwise $H_0$ is a good gadget of the third kind).
It also follows that $\vec{c}_1-\vec{1}(f)$ denotes the color profile of $\{u_1T_1,\ldots,u_\ell T_\ell\}$,
while $\vec{c}_2-\vec{1}(e)$ denotes the color profile of $\{v_1T_1,\ldots,v_\ell T_\ell\}$.
Note that by \eqref{equ:s-parts}, the type $u_iT_iv_i$ only depends on the parts that contain $u_i$ and $v_i$;
moreover, the common vertices in $V(e)\cap V(f)$ contribute equally to both $\vec{\#}(e)$ and $\vec{\#}(f)$.
So it can be deduced that
the difference between the color profiles of $\{u_1T_1,\ldots,u_\ell T_\ell\}$ and $\{v_1T_1,\ldots,v_\ell T_\ell\}$ equals $\vec{\#}(e)-\vec{\#}(f)$.
This says that
$$\big(\vec{c}_1-\vec{1}(f)\big)-\big(\vec{c}_2-\vec{1}(e)\big)=\vec{\#}(e)-\vec{\#}(f),$$
or equivalently, $$\vec{0}=\vec{c}_1-\vec{c}_2=\big(\vec{\#}(e)+\vec{1}(f)\big)-\big(\vec{\#}(f)+\vec{1}(e)\big).$$
This completes the proof of this claim.

Now we take two vertices $v_1\in V_1$ and $v_2\in V_2$.
Then $(v_1,v_2)$ is type $\mathbf{C_1}\mathbf{C_2}$.
So there exists some $T\in N_H(v_1)\cap N_H(v_2)$ such that $v_1Tv_2$ is $\mathbf{C_1}\mathbf{C_2}$.
That says, $v_1T$ is a $C_1$-edge and thus $\vec{1}(v_1T)=(1,0,\ldots,0)$.
Applying the previous claim on the edge $v_1T$,
we see that $$(a_1,\ldots,a_r)=\vec{a}=\vec{\#}(v_1T)-\vec{1}(v_1T),$$
where the first entry of $\vec{\#}(v_1T)$ is at least one as $v_1\in V_1$.
Therefore, all entries $a_i$ for $i\in [r]$ in $\vec{a}$ are non-negative and satisfy that $\sum_{i=1}^r a_i=k-1$.

We also can conclude that for all edges $e\in E(H)$, $\vec{\#}(e)$ only depends on the color of $e$.
Since there are exactly $r$ colors appearing, we will have $r$ distinct forms for $\vec{\#}(e)$, each differing at one entry.
This shows that $s=r$.
Finally, each $C_i$-edge $e$ in $H$ satisfies $\vec{\#}(e)=\vec{a}+\vec{1}(e)$, implying that
\[\mbox{$|V(e)\cap V_i|=a_i+1$ and $|V(e)\cap V_j|=a_j$ for every $j\in [r]\backslash \{i\}$.}\]

Analogously, if \eqref{equ:s-parts2} happens, then $s=r$, and there exists integers $b_1, \ldots, b_r\geq 1$ such that $\sum_{i=1}^r b_i = k + 1$ and
\[\mbox{$|V(e)\cap V_i|=b_i-1$ and $|V(e)\cap V_j|=b_j$ for every $j\in [r]\backslash \{i\}$.}\]
This completes the proof of Lemma~\ref{lem:3vtx}.
\end{proof}

Now using the item (C) of Lemma~\ref{lem:key}, we can rapidly derive the validity of Lemma~\ref{lem:key1}.

\subsection{Minimum degree given by extremal structures}
In this subsection, we provide precise estimations for the minimum vertex degrees of the $k$-graphs $H(n,\vec{\mathbf{a}})$, $H(\vec{V},\vec{\mathbf{a}})$, $\tilde{H}(n,\vec{\mathbf{b}})$ and $\tilde{H}^*(\vec{V},\vec{\mathbf{b}})$.
For integers $x_i\geq 0$ and $m=\sum_{i=1}^r x_i$, we define $\binom{m}{x_1,\ldots,x_r}:=\frac{m!}{x_1!\ldots x_r!}$.

\begin{lem}\label{lem:smallest_deg}
Let $k\geq 3$ and $r\geq 2$ be integers. For any $\vec{\mathbf{a}}=(a_1,\ldots,a_r)\in \mathbb{N}^r_{k-1}$\footnote{Recall that in this definition, we assume that $a_1 + a_2+ \ldots + a_r = k-1$ and $0\leq a_1\leq a_2\leq \ldots \leq a_r$.},
the following hold.
\begin{itemize}
\item [(A).]
We have $\delta\big(H(n,\vec{\mathbf{a}})\big)=(1+o(1))\cdot \binom{n-1}{k-1}\cdot g_r(\vec{\mathbf{a}})$,
where $o(1)\to 0$ as $n\to \infty$,
\[
g_r(\vec{\mathbf{a}}):=\binom{k-1}{a_1,\ldots,a_r}\cdot \prod_{i=1}^r\left(\frac{ra_i+1}{rk}\right)^{a_i}\cdot \left(1+\frac{a_1}{ra_1+1}\sum_{i=2}^r\frac{ra_i+1}{a_i+1}\right).
\]
In particular, this implies that
\[
g_r(k)= \max_{
\vec{\mathbf{a}}\in \mathbb{N}^r_{k-1}} g_r(\vec{\mathbf{a}}).
\]

\item [(B).] Let $\vec{V}=(V_1,V_2,\ldots, V_r)$ be a partition of $n$ vertices such that $\frac{|V_i|}{n}=(1+\epsilon_i)\frac{ra_i+1}{rk}$. Then
\[
\mbox{by setting } \epsilon = \max_{1\leq i\leq r} |\epsilon_i|\leq \frac{1}{2k}, \mbox{ we have }
\delta\big(H(\vec{V},\vec{\mathbf{a}})\big)\leq (1+4k\epsilon+o(1))\cdot \binom{n-1}{k-1}\cdot g_r(\vec{\mathbf{a}}).
\]
\end{itemize}
\end{lem}

\begin{proof}
In this proof, we denote $H$ as $H(n,\vec{\mathbf{a}})$, $H(\vec{V},\vec{\mathbf{a}})$, and we define $V_1 \cup V_2 \cup \ldots \cup V_r$ as the default $r$-partition of $H$.
By definition, for any $e\in E(H)$, there exists some $i\in [r]$ such that $|e\cap V_i|=a_i+1$ and $|e\cap V_j|=a_j$ for all $j\in [r]\backslash \{i\}$.
Let $E_i(H)$ be the set of all edges $e$ in $H$ with $|e\cap V_i|=a_i+1$.
Let $n_i=|V_i|$ for $i\in [r]$.
Suppose that each $n_i\to \infty$ as $n\to \infty$.
Then in both cases, we have
\begin{align*}
\delta(H)&=\min_{v\in V(H)}\sum_{\ell=1}^r\sum_{e\in E_\ell(H)}1_{\{v\in e\}}\\
&=\min_{j\in[r]}\sum_{\ell=1;\ell\neq j}^r\left(\prod_{i=1;i\neq j,\ell}^r\binom{n_i}{a_i}\right)\binom{n_\ell}{a_\ell+1}\binom{n_j-1}{a_j-1}
+\left(\prod_{i=1;i\neq j}^r\binom{n_i}{a_i}\right)\binom{n_j-1}{a_j}\\
&=\min_{j\in[r]}\left(\prod_{i=1;i\neq j}^r\binom{n_i}{a_i}\right)\cdot \binom{n_j-1}{a_j}
\left(1+\frac{a_j}{n_j-a_j}{\sum_{\ell=1;\ell\neq j}^r\frac{n_\ell-a_\ell}{a_\ell+1}}\right),
\end{align*}
thus implying that (here $o(1)$ denotes that $o(1) \to 0$ as $n \to \infty$)
\begin{equation}\label{equ:delta-H}
\frac{\delta(H)}{\binom{n-1}{k-1}}=(1+o(1))\cdot\min_{j\in [r]}\binom{k-1}{a_1,a_2,\ldots,a_r}\prod_{i=1}^r\left(\frac{n_i}{n}\right)^{a_i}\cdot \left(1+\sum_{i=1;i\neq j}^r\frac{n_ia_j}{n_j(a_i+1)}\right).
\end{equation}

We first prove (A). In this case, $n$ is divisible by $rk$ and $n_i=\frac{ra_i+1}{rk}n$. Hence, we derive from \eqref{equ:delta-H}
\begin{align*}
\frac{\delta(H)}{\binom{n-1}{k-1}}
=(1+o(1))\cdot\min_{j\in [r]} \binom{k-1}{a_1,a_2,\ldots,a_r}\prod_{i=1}^r\left(\frac{ra_i+1}{rk}\right)^{a_i}\cdot \left(1+\frac{a_j}{ra_j+1}\sum_{i=1;i\neq j}^r\frac{ra_i+1}{a_i+1}\right).
\end{align*}
Recall that $0\leq a_1\leq a_2\leq \ldots \leq a_r$. To prove (A), it remains to show that for all $2\leq j\leq r$,
\begin{equation}\label{ineq_of_deg}
\frac{a_1}{ra_1+1}\sum_{i=2}^r\frac{ra_i+1}{a_i+1}
\leq \frac{a_j}{ra_j+1}\sum_{i=1;i\neq j}^r\frac{ra_i+1}{a_i+1}
\end{equation}
Observing that for $0\leq a\leq b$, it holds that $\frac{a}{ra+1}\leq \frac{b}{rb+1}$ and $\frac{ra+1}{a+1}\leq \frac{rb+1}{b+1}$,
we can deduce that
\[
\frac{a_1}{ra_1+1}\sum_{i=2;i\neq j}^r\left(\frac{ra_i+1}{a_i+1}-\frac{ra_1+1}{a_1+1}\right)\leq
\frac{a_j}{ra_j+1}\sum_{i=2;i\neq j}^r\left(\frac{ra_i+1}{a_i+1}-\frac{ra_1+1}{a_1+1}\right).
\]
Therefore, to show the validity of \eqref{ineq_of_deg}, it suffices to show the following stronger inequality:
\[
\frac{a_1}{ra_1+1}\left (\sum_{i=2}^r\frac{ra_i+1}{a_i+1}-\sum_{i=2;i\neq j}^r\big(\frac{ra_i+1}{a_i+1}-\frac{ra_1+1}{a_1+1}\big)\right)
\leq
\frac{a_j}{ra_j+1}\left (\sum_{i=1;i\neq j}^r\frac{ra_i+1}{a_i+1}-\sum_{i=2;i\neq j}^r\big(\frac{ra_i+1}{a_i+1}-\frac{ra_1+1}{a_1+1}\big)\right)
\]
which is equivalent to $\frac{a_1}{ra_1+1}\frac{ra_j+1}{a_j+1}+(r-2)\frac{a_1}{a_1+1}
\leq (r-1)\frac{a_j}{ra_j+1}\frac{ra_1+1}{a_1+1}.$
This is true because
\begin{align*}
\frac{a_1}{ra_1+1}\frac{ra_j+1}{a_j+1}+(r-2)\frac{a_1}{a_1+1}-(r-1)\frac{a_j}{ra_j+1}\frac{ra_1+1}{a_1+1}=-\frac{(r-1)(a_j-a_1)(1+(r-1)a_1+a_j)}{(a_1+1)(ra_1+1)(a_j+1)(ra_j+1)}
\leq 0.
\end{align*}

Now we prove (B). In this case, we have $n_i=(1+\epsilon_i)\frac{ra_i+1}{rk}n$ and $\epsilon = \max_{1\leq i\leq r}|\epsilon_i|$.
Note that $0 \leq \epsilon \leq \frac{1}{2k}\leq \frac{1}{2}$.
So we have
$\frac{1}{1-\epsilon}\leq 1+2\epsilon \leq (1+\epsilon)^2$ and $(1+\epsilon)^{k+2} \leq \exp(\epsilon(k+2)) \leq 1+(e-1)(k+2)\epsilon$,
where we make use of the following facts that $0\leq \epsilon(k+2)\leq \frac{k+2}{2k}\leq 1$ and $e^x\leq 1+(e-1)x$ for $0\leq x\leq 1$.
Using \eqref{equ:delta-H} again, we can derive that
\begin{align*}
\frac{\delta(H)}{(1+o(1))\binom{n-1}{k-1}}
&\leq \min_{j\in [r]}\binom{k-1}{a_1,a_2,\ldots,a_r}
\prod_{i=1}^r\left(\frac{ra_i+1}{rk}(1+\epsilon)\right)^{a_i}
\cdot \left(1+\sum_{i=1;i\neq j}^r\frac{(1+\epsilon)a_j(ra_i+1)}{(1-\epsilon)(ra_j+1)(a_i+1)}\right)\\
&\leq \min_{j\in [r]} \binom{k-1}{a_1,a_2,\ldots,a_r}(1+\epsilon)^{k-1}\prod_{i=1}^r\left(\frac{ra_i+1}{rk}\right)^{a_i}\cdot (1+\epsilon)^3\left(1+\sum_{i=1;i\neq j}^r\frac{a_j(ra_i+1)}{(ra_j+1)(a_i+1)}\right)\\
&=(1+\epsilon)^{k+2}\cdot g_r(\vec{\mathbf{a}})\leq (1+(e-1)(k+2)\epsilon)\cdot g_r(\vec{\mathbf{a}})\leq (1+4k\epsilon)\cdot g_r(\vec{\mathbf{a}}),
\end{align*}
where the equality holds because of \eqref{ineq_of_deg} and the definition of $g_r(\vec{\mathbf{a}})$.
This finishes the proof.
\end{proof}

\begin{lem}\label{lem:smallest_deg2}
Let $k\geq 3$ and $r\geq 2$ be integers. For any $\vec{\mathbf{b}} = (b_1,\ldots, b_r)\in \mathbb{N}^r_{k+1}$ such that $b_1\geq 1$, the following hold.
\begin{itemize}
\item [(A)]
we have
$\delta\big(H(n,\vec{\mathbf{b}})\big)=(1+o(1))\cdot \binom{n-1}{k-1}\cdot \tilde{g}_r(\vec{\mathbf{b}})$, where $o(1)\to 0$ when $n\to \infty$ and
\[
\tilde{g}_r(\vec{\mathbf{b}}):=\binom{k}{b_1-1,b_2,\ldots,b_r}\cdot \prod_{i=2}^r\left(\frac{rb_i-1}{rk}\right)^{b_i}\cdot
\left(\frac{rb_1-1}{rk}\right)^{b_1-1}\cdot \left(\frac{r(b_1-1)}{rb_1-1}+\sum_{i=2}^r\frac{rb_i}{rb_i-1}\right).
\]

\item [(B)]
Let $\vec{V}=(V_1,V_2,\ldots, V_r)$ be a partition of $n$ vertices such that $\frac{|V_i|}{n}=(1+\epsilon_i)\frac{rb_i-1}{rk}$. Then
\[
\mbox{by setting } \epsilon = \max_{1\leq i\leq r} |\epsilon_i|\leq \frac{1}{2k}, \mbox{ we have }
\delta\big(\tilde{H}(\vec{V},\vec{\mathbf{b}})\big)\leq (1+4k\epsilon+o(1))\cdot \binom{n-1}{k-1}\cdot \tilde{g}_r(\vec{\mathbf{b}}).
\]
\end{itemize}

\end{lem}
\begin{proof}
In this proof, we denote $H$ as $\tilde{H}(n,\vec{\mathbf{b}})$, $\tilde{H}(\vec{V},\vec{\mathbf{b}})$, and we define $V_1 \cup V_2 \cup \ldots \cup V_r$ as the default $r$-partition of $H$.
By definition, for any $e\in E(H)$, there exists some $i\in [r]$ such that $|e\cap V_i|=b_i-1$ and $|e\cap V_j|=b_j$ for all $j\in [r]\backslash \{i\}$.
Let $E_i(H)$ be the set of all edges $e$ in $H$ with $|e\cap V_i|=b_i-1$.
Let $n_i=|V_i|$ for $i\in [r]$.
Suppose that each $n_i\to \infty$ as $n\to \infty$.
Then in both cases, we have\footnote{We admit that $\binom{n}{m}=0$ for $m<0$ or $m > n$.}.
\begin{align*}
\delta(H) &= \min_{v\in V(H)}\sum_{\ell=1}^r \sum_{e\in E_\ell(H)}1_{v\in e}\\
&= \min_{j\in [r]} \sum_{\ell=1; \ell\neq j}^r \left(\prod_{i=1;i\neq j,\ell}^r \binom{n_i}{b_i}\right)\binom{n_\ell}{b_\ell -1}\binom{n_j-1}{b_j-1} + \left(\prod_{i=1;i\neq j}^r \binom{n_i}{b_i}\right)\binom{n_j-1}{b_j-2}\\
&=\left(\prod_{i=1;i\neq j}^r \binom{n_i}{b_i}\right)\binom{n_j-1}{b_j}\left(\frac{b_j(b_j-1)}{(n_j-b_j+1)(n_j-b_j)}+\frac{b_j}{n_j-b_j}\sum_{\ell=1;\ell\neq j}^r\frac{b_\ell}{n_\ell-b_\ell + 1}\right),
\end{align*}
thus implying that (here $o(1)$ denotes that $o(1)\to 0$ as $n\to \infty$)
\begin{align}\label{equ:delta-tildeH}
&\frac{\delta(H)}{(1+o(1))\binom{n-1}{k-1}}\\
=~&\frac{1}{k(k+1)}\cdot\min_{j\in[r]}
\binom{k+1}{b_1,b_2,\ldots, b_r}\left(\prod_{i=1}^r\left(\frac{n_i}{n}\right)^{b_i}\right)\left(b_j(b_j-1)\cdot\left(\frac{n}{n_j}\right)^2
+ \sum_{\ell=1,\ell\neq j}^rb_jb_\ell\cdot \left(\frac{n^2}{n_jn_\ell}\right)\right).
\end{align}

We first prove (A). In this case, $n$ is divisible by $rk$ and $n_i=\frac{rb_i-1}{rk}n$. Hence, we derive from \eqref{equ:delta-tildeH}
\begin{align*}
\frac{\delta(H)}{(1+o(1))\binom{n-1}{k-1}}
=&\frac{1}{k(k+1)}\cdot\min_{j\in [r]} \binom{k+1}{b_1,b_2,\ldots,b_r}\prod_{i=1}^r
\left(\frac{rb_i-1}{rk}\right)^{b_i}\\
&\cdot\left(b_j(b_j-1)\cdot\left(\frac{rk}{rb_j-1}\right)^2+\sum_{\ell=1;\ell\neq j}^rb_jb_\ell\cdot \left(\frac{rk}{rb_j-1}\right)\left(\frac{rk}{rb_\ell-1}\right)\right)\\
=&\frac{r^2k}{k+1}\cdot \min_{j\in [r]} \binom{k+1}{b_1,b_2,\ldots,b_r}\prod_{i=1}^r
\left(\frac{rb_i-1}{rk}\right)^{b_i}
\left(\frac{b_j}{rb_j-1}\left(\frac{b_j-1}{rb_j-1}+ \sum_{\ell=1;\ell\neq j}^r\frac{b_\ell}{rb_\ell-1}\right)\right).
\end{align*}
Recall that $1\leq b_1\leq b_2\leq \ldots \leq b_r$. To prove (A), it remains to show that for all $2\leq j\leq r$,
\begin{equation}\label{ineq_of_deg2}
\frac{b_1}{rb_1-1}\left(
\frac{b_1-1}{rb_1-1}+\sum_{\ell=2}^r\frac{b_\ell}{rb_\ell - 1}
\right)
\leq \frac{b_j}{rb_j-1}\left(
\frac{b_j-1}{rb_j-1}+\sum_{\ell=1;\ell\neq j}^r\frac{b_\ell}{rb_\ell - 1}
\right)
\end{equation}
Observing that for $1\leq a\leq b$, it holds that $\frac{a}{ra-1}\geq \frac{b}{rb-1}$,
we can deduce that
\begin{align*}
&\frac{b_1}{rb_1-1}\left(
\frac{b_1-1}{rb_1-1}+\sum_{\ell=2}^r\frac{b_\ell}{rb_\ell - 1}
\right)- \frac{b_j}{rb_j-1}\left(
\frac{b_j-1}{rb_j-1}+\sum_{\ell=1;\ell\neq j}^r\frac{b_\ell}{rb_\ell - 1}\right)\\
\leq~ & \frac{b_1(b_1-1)}{(rb_1-1)^2}-\frac{b_j(b_j-1)}{(rb_1-1)(rb_j-1)}+(r-2)\left(\frac{b_1}{rb_1-1}-\frac{b_j}{rb_j-1}\right)\frac{b_1}{rb_1-1}\\
=~ &\frac{(b_1-b_j)((r-1)b_1 + b_j-1)}{(rb_1-1)^2(rb_j-1)^2}\leq 0.
\end{align*}
Hence
\begin{align*}
\frac{\delta(H)}{\binom{n-1}{k-1}}
&= (1+o(1))\frac{r^2k}{k+1}\binom{k+1}{b_1,b_2,\ldots, b_r}\prod_{i=1}^r\left(\frac{rb_i-1}{rk}\right)^{b_i}\left(\frac{b_1}{rb_1-1}\left(\frac{b_1-1}{rb_1-1}+\sum_{i=2}^r\frac{b_i}{rb_i-1}\right)\right)\\
&=(1+o(1))\cdot\binom{k}{b_1-1,b_2,\ldots,b_r}\prod_{i=2}^r\left(\frac{rb_i-1}{b_i}\right)^{b_i}\cdot\left(\frac{rb_1-1}{rk}\right)^{b_1-1}\cdot\left(\frac{r(b_1-1)}{rb_1-1}+\sum_{i=2}^r\frac{rb_i}{rb_i-1}\right)
\end{align*}
finishing the proof of item (A).

Now we prove (B). In this case, we have $n_i=(1+\epsilon_i)\frac{rb_i-1}{rk}n$ and $\epsilon = \max_{1\leq i\leq r}|\epsilon_i|$.
Note that $0 \leq \epsilon \leq \frac{1}{2k}\leq \frac{1}{2}$.
So we have
$\frac{1}{1-\epsilon}\leq 1+2\epsilon \leq (1+\epsilon)^2$ and $(1+\epsilon)^{k+2} \leq \exp(\epsilon(k+2)) \leq 1+(e-1)(k+2)\epsilon$,
where we make use of the following facts that $0\leq \epsilon(k+2)\leq \frac{k+2}{2k}\leq 1$ and $e^x\leq 1+(e-1)x$ for $0\leq x\leq 1$.
Using \eqref{equ:delta-tildeH} again, let $\delta_{ij}$ be the Kronecker notation such that $\delta_{ij}=1$ if $i=j$ and $\delta_{ij}=0$ if $i\neq j$, we can derive that
\begin{align*}
&\frac{\delta(H)}{(1+o(1))\binom{n-1}{k-1}}\\
=~ & \frac{1}{k(k+1)}\cdot\min_{j\in[r]}
\binom{k+1}{b_1,b_2,\ldots, b_r}\left(\prod_{i=1}^r\left(\frac{n_i}{n}\right)^{b_i}\right)\left(b_j(b_j-1)\cdot\left(\frac{n}{n_j}\right)^2
+ \sum_{\ell=1,\ell\neq j}^rb_jb_\ell\cdot \left(\frac{n^2}{n_jn_\ell}\right)\right)\\
=~ & \frac{1}{k}\cdot\min_{j\in[r]}
\binom{k}{b_1,\ldots, b_j-1,\ldots, b_r}\left(\prod_{i=1}^r\left(\frac{n_i}{n}\right)^{b_i-\delta{ij}}\right)\left((b_j-1)\cdot\left(\frac{n}{n_j}\right)+\sum_{\ell=1;\ell\neq j}^{r}b_\ell\cdot\left(\frac{n}{n_\ell}\right)\right)\\
\leq~& \min_{j\in[r]}
\binom{k}{b_1,\ldots, b_j-1,\ldots, b_r}\prod_{i=1}^r\left(\frac{rb_i-1}{rk}(1+\epsilon)\right)^{b_i-\delta{ij}}\\
&\quad\ \cdot\left((b_j-1)\cdot\frac{1}{k}\cdot\left(\frac{rk}{(1-\epsilon)(rb_j-1)}\right)
+\sum_{\ell=1;\ell\neq j}^{r}b_\ell\cdot\frac{1}{k}\cdot\left(\frac{rk}{(1-\epsilon)(rb_\ell-1)}\right)\right)\\
=~& \min_{j\in [r]}\binom{k}{b_1,\ldots,b_j-1,\ldots,b_r}
\prod_{i=1}^r\left(\frac{rb_i-1}{rk}(1+\epsilon)\right)^{b_i-\delta_{ij}}
\cdot \frac{1}{1-\epsilon}\cdot\left(\frac{r(b_j-1)}{rb_j-1}+\sum_{\ell=1;\ell\neq j}^{r}\frac{rb_\ell}{rb_\ell-1}\right)\\
\leq~& \frac{(1+\epsilon)^k}{1-\epsilon}\cdot\tilde{g}_r(\vec{\mathbf{b}})
\leq (1+\epsilon)^{k+2}\cdot \tilde{g}_r(\vec{\mathbf{b}})
\leq (1+(e-1)(k+2)\epsilon)\cdot \tilde{g}_r(\vec{\mathbf{b}})\leq (1+4k\epsilon)\cdot 
\tilde{g}_r(\vec{\mathbf{b}}),
\end{align*}
where the equality holds because of \eqref{ineq_of_deg2} and the definition of $\tilde{g}_r(\vec{\mathbf{b}})$.
This finishes the proof.
\end{proof}

In the following lemma, we will see that for $r\geq 3$ and $k\geq 3$, $\lim_{n\to\infty} \frac{\delta(H(n,\vec{\mathbf{b}}))}{\binom{n-1}{k-1}} < \max\{f(k),g_r(k)\}$.
Furthermore, it suggests that $\tilde{H}^*(n,\vec{\mathbf{b}})$ will not be the extreme graph with low perfect matching discrepancy and high degree. The proof can be found in Appendix \ref{sec:type2}.

\begin{lem}\label{lem:type2}
For any $r\geq 3$, $k\geq 3$ and $\vec{\mathbf{b}}\in\mathbb{N}_{k+1}^r$ satisfying $b_1\geq 1$, then the following hold:
\begin{itemize}
\item [(A)] If $b_1=1$, then for every positive integer $n$ is divisible by $rk$, $\tilde{H}(n,\vec{\mathbf{b}})$ is a subgraph of\\ $H(n,(0,\ldots,0,k-1))$, implying that $\tilde{g}_r(\vec{\mathbf{b}}) \leq g_r(k)\leq \max\{f(k), g_r(k)\}$;
\item [(B)] If $b_1\geq 2$, then
\[
\tilde{g}_r(\vec{\mathbf{b}})< f(k) \leq \max\{f(k), g_r(k)\},
\]
\end{itemize}

\end{lem}

\section{Proof of Theorem \ref{main}}
For fixed integers $k\geq 3$ and $r\geq 2$, we take $\gamma=\gamma(k,r) > 0$ to be sufficiently small.\footnote{Our proof shows that $\gamma$ can be chosen as $\gamma=\Theta(1/r^2k^2)$.}
Given any $\eta>0$,\footnote{We may assume that $0<\eta<\frac12$, because otherwise the inequality regarding the minimum vertex degree of $H$ below would not hold.} let $n_0=n_0(r,k,\eta)$ be sufficiently large.
Let $H$ be a $k$-uniform hypergraph on $n\geq n_0$ vertices with an $r$-edge-coloring $c: E(H)\rightarrow\{C_1,C_2,\ldots,C_r\}$ such that $n\equiv 0\pmod k$ and \[\delta(H)>\big(\max\{f(k), g_r(k)\}+\eta\big)\cdot \binom{n-1}{k-1}.\]
Suppose for a contradiction that every perfect matching of $H$ has less than $\frac{n}{rk}(1+\gamma\cdot\eta)$ edges with the color $C_i$ for any $i\in [r]$.
For a subgraph $F$ of $H$, let $C_i(F)$ denote the number of $C_i$-edges contained in $F$.
As a consequence, we can derive that for every perfect matching $\mathcal{M}$ of $H$ and every $i\in [r]$,
\begin{equation}\label{equ:cp_i}
\frac{n}{rk}\bigg(1-(r-1)\gamma\cdot\eta\bigg)< C_i(\mathcal{M}) <\frac{n}{rk}\bigg(1+\gamma\cdot\eta\bigg).
\end{equation}

First we consider when $H$ contains $t=\eta n/\big(20\cdot rk^2(k+1)\big)$ (vertex-)disjoint good gadgets (of any kind),
denoted by $G_1,\ldots, G_t$.
For $j\in [t]$, let $\mathcal{M}_{j}$ and $\mathcal{N}_{j}$ be two perfect matchings of $G_j$ with different color profiles.
Let $H'=H-\cup_{j=1}^t V(G_j)$.
Since each $G_i$ has at most $k^2+k$ vertices,
we have
\begin{equation}\label{equ:delta(H')}
\delta(H')\geq \delta(H)-t(k^2+k)\binom{n-2}{k-2}>\big(\max\{f(k), g_r(k)\}+\eta/2\big)\cdot\binom{n-1}{k-1}.
\end{equation}
Note that $|V(H')|$ is divisible by $k$. By the definition of $f(k)$, $H'$ has a perfect matching say $\mathcal{L}$.
Since every $\mathcal{M}_{j}$ and $\mathcal{N}_{j}$ have different color profiles,
by average there exists some $i\in [r]$ (let us say $i=1$) such that
at least $t/r$ many indices $j\in [t]$ satisfy $C_1(\mathcal{M}_{j})\neq C_1(\mathcal{N}_{j})$.
Renaming if needed we may assume that for all $j\in [t]$, we have $C_1(\mathcal{M}_{j})\geq C_1(\mathcal{N}_{j})$,
and thus we obtain that
$$\sum_{j=1}^t C_1(\mathcal{M}_{j})\geq \sum_{j=1}^t C_1(\mathcal{N}_{j})+\frac{t}{r}.$$
Let $\mathcal{M}=\mathcal{L}\cup \mathcal{M}_1\cup \ldots \cup \mathcal{M}_t$ and $\mathcal{N}=\mathcal{L}\cup \mathcal{N}_1\cup \ldots \cup \mathcal{N}_t$ be two perfect matchings of $H$.
Then
$$C_1(\mathcal{M})-C_1(\mathcal{N})=\sum_{j=1}^t \Big(C_1(\mathcal{M}_{j})-C_1(\mathcal{N}_{j})\Big)\geq \frac{t}{r}=\frac{\eta n}{20\cdot r^2k^2(k+1)}.$$
On the other hand, using \eqref{equ:cp_i} we have $C_1(\mathcal{M})-C_1(\mathcal{N})\leq \frac{\gamma}{k}\eta n$,
which is a contradiction as $\gamma=\gamma(k,r)$ is chosen to be sufficiently small.
This proves the first case.

Now we may assume that $H$ contains less than $\eta n/\big(20\cdot rk^2(k+1)\big)$ disjoint good gadgets.
We can greedily choose $t$ disjoint good gadgets, denoted by $G_1,\ldots,G_t$, such that $H':= H-\cup_{i=1}^t V(G_i)$ contains no good gadgets of any kind, where $t< \eta n/\big(20\cdot rk^2(k+1)\big)$.
Let $m=|\cup_{i=1}^t V(G_i)|$. So $$m\leq t(k^2+k)<\frac{\eta n}{20\cdot rk}.$$
Also note that $|V(H')|$ is divisible by $k$ and moreover, the same inequality \eqref{equ:delta(H')} holds for $H'$.
Hence, $H'$ contains at least one perfect matching.
Let $\mathcal{L}$ be any perfect matching of $H'$.
If $\mathcal{L}$ contains at least $\frac{|V(H')|}{rk}+\frac{m}{rk}+\frac{\gamma \eta n}{rk}$ edges of the same color,
then, when combined with any perfect matching of $G_i$ for all $i\in [t]$, it forms a perfect matching of $H$ which contradicts \eqref{equ:cp_i}.
Hence we may assume that $$C_i(\mathcal{L})<\frac{|V(H')|}{rk}+\left(\frac{m}{rk}+\frac{\gamma \eta n}{rk}\right)\leq \frac{|V(H')|}{rk}+\frac{\eta \cdot |V(H')|}{15\cdot r^2k^2} \mbox{ for any } i\in [r] ,$$
where the last inequality holds because $m<\frac{\eta n}{20\cdot rk}$ and $\gamma$ is sufficiently small.
It further implies that
\begin{equation}\label{equ:Ci(L)}
\frac{|V(H')|}{rk}-\frac{(r-1)\eta \cdot |V(H')|}{15\cdot r^2k^2}<C_i(\mathcal{L})< \frac{|V(H')|}{rk}+\frac{\eta \cdot |V(H')|}{15\cdot r^2k^2} \mbox{ for each } i\in [r].
\end{equation}
In particular, it shows that $H'$ has at least one $C_i$-edge for each $i\in [r]$.
Using \eqref{equ:delta(H')}, it is also easy to derive that $\delta(H')>\frac{1}{2}\binom{n-1}{k-1}+\frac{k^2+1}{2}\binom{n-2}{k-2}$.
Now we see that $H'$ satisfies all conditions of Lemma~\ref{lem:key1}.
By Lemma~\ref{lem:key1}, we first consider that there exists an $r$-partition $\vec{V}'=(V_1',\ldots, V_r'))$ of $V(H')$ and a vector $\vec{\mathbf{a}}\in\mathbb{N}^r_{k-1}$ such that $H'$ is an $r$-edge-colored subgraph of $H^*(\vec{V}',\vec{\mathbf{a}})$.
Any perfect matching $\mathcal{L}$ of $H'$ is of course a perfect matching of $H^*(\vec{V}',\vec{\mathbf{a}})$.
Now recall the observation \eqref{equ:obs2}, which asserts that $$|V_i'|=a_i\frac{|V(H')|}{k}+C_i(\mathcal{L}) \mbox{ for each } i\in [r].$$
Let $|V_i'|/|V(H')|=(1+\epsilon_i)\frac{ra_i+1}{rk}$.
Then the above equality, along with \eqref{equ:Ci(L)}, implies that
\begin{equation}\label{equ:eps}
\epsilon:= \max_{1\leq i\leq r} |\epsilon_i|\leq \max_{1\leq i\leq r} \frac{(r-1)\eta}{15\cdot rk\cdot (ra_i+1)}\leq \frac{\eta}{15k}\leq \frac{1}{2k},
\end{equation}
where we use $\eta<\frac12$.
By item (B) of Lemma~\ref{lem:smallest_deg}, since $n$ is sufficiently large, we can derive that
\begin{align*}
\delta(H')&\leq \delta\big(H(\vec{V}',\vec{\mathbf{a}})\big)\leq (1+4k\epsilon+\frac{\eta}{15})\cdot \binom{|V(H')|-1}{k-1}\cdot g_r(\vec{\mathbf{a}})\\
&\leq \left(1+\frac{\eta}{3}\right)\cdot \binom{n-1}{k-1}\cdot g_r(k)\leq \left(g_r(k)+\frac{\eta}{3}\right)\cdot \binom{n-1}{k-1},
\end{align*}
where the second last inequality holds because of \eqref{equ:eps} and the definition of $g_r(k)$, and the last inequality follows from the fact that $g_r(k)\leq 1$.
On the other hand, by \eqref{equ:delta(H')} we have $\delta(H')\geq (g_r(k)+\frac{\eta}2)\cdot\binom{n-1}{k-1}$, which is a contradiction.

Now we consider the case of $H'$ is a $r$-edge-colored subgraph of $\tilde{H}^*(\vec{V}',\vec{\mathbf{b}})$ for some partition $\vec{V}'$ of $V'$ and $\vec{\mathbf{b}}\in\mathbb{N}_{k+1}^r$ such that $b_1\geq 1$. 
If $r=2$, then $\tilde{H}((V_1, V_2),(b_1,b_2))= H((V_1, V_2),(b_1-1,b_2-1))$.
Hence it leads to a contradiction by previous discussion.
So we may assume that $r\geq 3$.
Then by item (B) of Lemma \ref{lem:smallest_deg2} and Lemma \ref{lem:type2}, following the same discussion, we have 
\[
\delta(H')\leq \left(\tilde{g}_r(b_1,\ldots,b_r)+\frac{\eta}{3}\right) \cdot \binom{n-1}{k-1}
\leq \left(\max\{f(k), g_r(k)\}+\frac{\eta}{3}\right) \cdot \binom{n-1}{k-1},
\]
which contradicts with $\delta(H')\geq (\max\{f(k), g_r(k)\}+\frac{\eta}2)\cdot\binom{n-1}{k-1}$.
This completes the proof of Theorem \ref{main}. \qed

\section{Proof of Theorem~\ref{thm:main2}}
The goal of this section is to establish the proof of Theorem \ref{thm:main2}.
Throughout this section, we denote the function $f_0(k)=1-\left(1-\frac{1}{k}\right)^{k-1}$. Recall that we have
\begin{equation}\label{equ:f>f0}
f(k)=\limsup_{n\to\infty}\frac{m(k,n)}{\binom{n-1}{k-1}}\geq f_0(k) \mbox{ ~~ and ~~  } g_r(k)= \max_{\vec{\mathbf{a}}\in \mathbb{N}^r_{k-1}} g_r(\vec{\mathbf{a}}),
\end{equation}
where the first inequality becomes an equality when $k\in \{2,3,4,5\}$, and the second expression is given by Lemma~\ref{lem:smallest_deg}.

In the following lemma, we present some sufficient conditions under which $f(k)>g_r(k)$ holds.
The proof can be found in Appendix \ref{sec:g<k}.

\begin{lem}\label{lem:g<k}
Let $k\geq 3$ and $r\geq 2$ be integers. Then $f(k)>g_r(k)$ if one of the following holds:
\begin{itemize}
\item $r=2$ and $k\geq 20$;
\item $r\geq 3$ and $k\geq 10$; and
\item $r\geq 6$ and $k\geq 3$.
\end{itemize}
\end{lem}

Using this lemma, we are ready to present the proof of  Theorem~\ref{thm:main2}.

\begin{proof}[\bf Proof of Theorem~\ref{thm:main2}.]
By Proposition~\ref{prop:main} and Theorem~\ref{main}, we can derive that $h_r(k)= \max\{f(k), g_r(k)\}$ for all $k\geq 3$  and $r\geq 2$.
By Lemma~\ref{lem:g<k}, we see that $f(k)>g_r(k)$ holds whenever one of the following holds: (I) $r=2$ and $k\geq 20$; (II) $r\geq 3$ and $k\geq 10$, and (III) $r\geq 6$ and $k\geq 3$.

To complete the proof, it suffices to examine the following uncovered cases:
(A) $r=2$ and $3\leq k\leq 19$ and (B) $3\leq r\leq 5$ and $3\leq k\leq 9$.
These cases are finite in number, allowing us to determine the precise values of the corresponding constants.
Let $f_0(k)=1-(1-\frac{1}{k})^{k-1}$.
Recall that $f(k)\geq f_0(k)$, with equality if $k\in \{3,4,5\}$.
In the following two tables, we compare with the approximate values of $g_r(k)$ and $f_0(k)$ for the cases (A) and (B), respectively.
Numerical details for Table~\ref{tab:case4} and Table~\ref{tab:case5} can be found in Appendix~\ref{table}.

\begin{table}[ht]
    \centering
    \begin{tabular}{c|ccccccccc}
    \toprule
     $k$  & 3 &4 &5 & 6& 7& 8& 9&10&11 \\
    \midrule
    $g_2(k)$  & \textcolor{green}{0.75}\tablefootnote{The value of $g_2(3)$ is maximized by the vector $\vec{a}=(1,1)$.} &\textcolor{green}{0.6836}\tablefootnote{The value of $g_2(4)$ is maximized by the vector $\vec{a}=(1,2)$.} &\textcolor{green}{0.6561}\tablefootnote{The value of $g_2(5)$ is maximized by the vector $\vec{a}=(0,4)$.} & 0.6472& 0.6410& 0.6365& 0.6330&0.6302&0.6280 \\
    $f_0(k)$  & 0.5556 &0.5781 &0.5904 & 0.5981& 0.6034& 0.6073& 0.6103&0.6126&0.6145 \\
    \midrule[0.8pt]
    $k$  & 12 &13 &14 & 15&16& 17& 18&19& \\
    \midrule
    $g_2(k)$  & 0.6262 &0.6246 &0.6233 & 0.6221&0.6211& \textcolor{blue}{0.6202}& \textcolor{blue}{0.6195}&\textcolor{blue}{0.6188}& \\
    $f_0(k)$  & 0.6160 &0.6173 &0.6184 & 0.6194&0.6202& 0.6209& 0.6216&0.6221& \\
    \bottomrule
    \end{tabular}
    \caption{The values of $g_2(k)$ and $f_0(k)$ for $3\leq k\leq 19$}
    \label{tab:case4}
\end{table}

\begin{table}[ht]
    \centering
    \begin{tabular}{c|ccccccc}
    \toprule
       $k$  & 3 &4 &5 & 6& 7& 8& 9 \\
         \midrule
    $g_3(k)$   & \textcolor{red}{0.6049}\tablefootnote{The value of $g_3(3)$ is maximized by the vector $\vec{a}=(0,0,2)$.} &\textcolor{red}{0.5787}\tablefootnote{The value of $g_3(4)$ is maximized by the vector $\vec{a}=(0,0,3)$.} & 0.5642& 0.5549&0.5485 & 0.5439& 0.5403 \\
    $g_4(k)$     & \textcolor{red}{0.5625}\tablefootnote{The value of $g_4(3)$ is maximized by the vector $\vec{a}=(0,0,0,2)$.} & 0.5363& 0.5220&0.5129 &0.5066 & 0.5020&0.4985  \\
    $g_5(k)$     & 0.5378 &0.512 &0.4979 &0.4889 &0.4828 &0.4783 & 0.4749 \\
    $f_0(k)$  &0.5556  & 0.5781 &0.5904 & 0.5981& 0.6034& 0.6073& 0.6103  \\
    \bottomrule
    \end{tabular}
    \caption{The values of $g_r(k)$ and $f_0(k)$ for $3\leq r\leq 5$ and $3\leq k\leq 9$}
    \label{tab:case5}
\end{table}

In Table \ref{tab:case4}, the values of $g_2(k)$ for $3 \leq k \leq 5$ (highlighted in green) indicate that $g_2(k) > f_0(k) = f(k)$,
the values of $g_2(k)$ for $17 \leq k \leq 19$ (highlighted in blue) indicate that $f(k)\geq f_0(k) > g_2(k)$,
and the other values present situations where a comparison between $g_2(k)$ and $f(k)$ is infeasible.
In Table \ref{tab:case5}, the values of $g_r(k)$ for $(r,k)\in \{(3,3),(3,4),(4,3)\}$ (highlighted in red) indicate that $g_r(k) > f_0(k) = f(k)$,
while all other values of $g_r(k)$ demonstrate that $f(k)\geq f_0(k)>g_r(k)$.
In particular, each footnote indicates the unique maximum vector $\vec{a}$ for the corresponding $g_r(k)= \max_{\vec{\mathbf{a}}\in \mathbb{N}^r_{k-1}} g_r(\vec{\mathbf{a}})$.

Taken together with all the information above, this leads to the assertions in Theorem~\ref{thm:main2} for $f(k) > g_r(k)$ and $g_r(k) > f(k)$, respectively.
\end{proof}

\section{Concluding remarks}
In this paper, we establish the minimum vertex degree threshold of perfect matchings with high $r$-color discrepancy in $k$-uniform hypergraphs, by showing that $h_r(k)= \max\{f(k), g_r(k)\}$ for all $k\geq 3$ and $r\geq 2$.
Our proofs have the potential to unveil further insights, and we would like to offer some additional remarks on this matter.

\begin{itemize}
\item It should be noted that if the matching existence conjecture (i.e., \eqref{equ:f(k)_ell} is conjectured to be an equality) holds true for $6\leq k\leq 16$, our proofs would lead to the conclusion that $h_2(k)=g_2(k)$ for $6\leq k\leq 16$. In these 11 instances, in addition to the 6 cases outlined in Theorem~\ref{thm:main2} under the condition $g_r(k)> f(k)$, the conclusion $h_r(k)=g_r(k)$ would exclusively apply, while all other cases would follow $h_r(k)=f(k)$.

\item In all cases where $g_r(k)>f(k)$ occurs, our proofs can be easily modified into a stability result, with the corresponding $H^*(n,\vec{\mathbf{a}})$ as the $r$-edge-colored extremal hypergraph, where $\vec{a}$ denotes the unique maximum vector for the optimization $g_r(k)= \max_{\vec{\mathbf{a}}\in \mathbb{N}^r_{k-1}} g_r(\vec{\mathbf{a}})$.

\item The key technical Lemma~\ref{lem:key1} (as well as other lemmas in Subsection~\ref{subsec:local} under the good-gadget-free condition) holds without requiring the existence of perfect matchings or the condition $n\equiv 0\pmod k$.
So our proofs can also help illuminate the establishment of the corresponding threshold for a {\it near-perfect matching} (i.e., a matching of size $\frac{n}{k}-O_k(1)$ in $n$-vertex $k$-graphs) with high discrepancy.
    For a more in-depth discussion, we recommend readers refer to the pertinent paragraph in the concluding section of Gishboliner-Glock-Sgueglia \cite{GGS}.

\item We have seen that in $k$-graphs for $k\geq 3$, the $\ell$-degree discrepancy threshold $h^\ell_r(k)$ equals the $\ell$-degree existence threshold $f^\ell(k)$ of perfect matchings for all but a finite number of cases.
Recall the well-known matching conjecture that $f^\ell(k)=\max\left\{\frac12, 1-\left(1-\frac1k\right)^{k-\ell}\right\}.$
Constructions that result in the constant $\frac12$ are commonly referred to as {\it parity obstacles}.
Also recall the definition of $m_\ell(k, n)$.
We wonder if in the case that $r=2$, $\ell=k-1$ and $k\geq 5$ (note that $h_2^{k-1}(k)=f^{k-1}(k)=\frac12$ is determined by the parity obstacle), there exist positive constants $c=c(k,\ell)$ and $C=C(k,\ell)$ such that for sufficiently large $n$ with $n\equiv0\pmod k$, any $2$-edge-colored $n$-vertex $k$-graph $H$ with $\delta_{k-1}(H)\geq m_{k-1}(k, n)+C$ contains a perfect matching with at least $\frac{n}{2k}+c\cdot n$ edges with the same color.
In essence, we suspect that in these cases, once a perfect matching emerges, a perfect matching with high discrepancy appears almost immediately.
\end{itemize}

\medskip

\noindent {\bf Acknowledgements.}
After submitting the first version of our draft to arXiv, we learned that H. H\'an, R. Lang, J. P. Marciano, M. Pavez-Sign\'e, N. Sanhueza-Matamala, A. Treglown, and C. Z\'arate-Guer\'en independently and simultaneously proved the same main results \cite{HLM24}.
We would like to thank N. Sanhueza-Matamala for pointing out a missing case in the proof of Lemma~12 (B) in the first version, which prompted updates in this new version.
We also appreciate A. Treglown for providing a 3-uniform counterexample to a question we raised in the final point of our concluding remarks in the first version of this paper,
which we have addressed accordingly.

\appendix
\appendixpage

\section{Proof of Lemma~\ref{lem:type2}}\label{sec:type2}
In this section, we give a detailed proof of Lemma \ref{lem:type2}.

We will first establish some basic inequalities in the following two propositions.
These propositions will be frequently used in the numerical proofs in this paper.
Throughout, we use $\exp(x)$ to denote the exponential function $e^x$.

\begin{prop}\label{prop:basic_ineq}
Let $k$ be a positive integer. Then the following hold:
\begin{itemize}
\item[(a)] For $k\geq 3$, $(1-\frac{1}{k})^{k-1}\leq (1-\frac{1}{3})^2=\frac{4}{9}\leq 1.21 e^{-1}$;
\item[(b)] For $k\geq 10$, $(1-\frac{1}{k})^{k-1}\leq (1-\frac{1}{10})^9\leq 1.06 e^{-1}$;
\item[(c)] For $k\geq 20$, $(1-\frac{1}{k})^{k-1}\leq (1-\frac{1}{20})^{19}\leq 0.378$, $(1-\frac{3}{2k})^{k-2}\leq (1-\frac{3}{40})^{18}\leq 0.246$
and $(1-\frac{1}{2k})^{k-1}\leq (1-\frac{1}{40})^{19}\leq 0.619$;
\item[(d)] For $k\geq 3$, if $r\geq 3$ and $1\leq m\leq k-2$ are integers, then
\[
\sqrt{\frac{k-1}{k-m-1}}\left(\frac{1}{2\pi}\right)^{\frac{m}{2}}
\leq \sqrt{2m}\left(\frac{1}{2\pi}\right)^{\frac{m}{2}}
\leq \exp\left(-\frac{m}{3}\right)\leq \exp\left(-\frac{m}{r}\right).
\]
\end{itemize}
\end{prop}

\begin{proof}
We first show that $h_1(k) = (1-\frac{1}{k})^{k-1}$ is decreasing over $\mathbb{N}^+$.
Indeed, it follows as
\[
\frac{h_1(k)}{h_1(k+1)}=\frac{(k^2-1)^{k-1}(k+1)}{k^{2k-1}}=\frac{(1-\frac{1}{k^2-1})^{k-1}}{1-\frac{1}{k+1}}\geq \frac{1-\frac{k-1}{k^2-1}}{1-\frac{1}{k+1}}=1.
\]
Hence by simple calculation, we can directly derive (a), (b) and the first assertion of (c).
Similarly, we can show that
$h_2(k)=(1-\frac{3}{2k})^{k-2}$ and $h_3(k)=(1-\frac{1}{2k})^{k-1}$ are decreasing for $k\geq 2$:
\begin{align*}
\frac{h_2(k)}{h_2(k+1)}=\frac{2k+2}{2k-1}\left(\frac{4k^2-2k-6}{4k^2-2k}\right)^{k-2}=\frac{(1-\frac{3}{k(2k-1)})^{k-2}}{1-\frac{3}{2k+2}}
\geq \frac{1-\frac{3(k-2)}{k(2k-1)}}{1-\frac{3}{2k+2}}\geq 1
\end{align*}
\begin{align*}
\frac{h_3(k)}{h_3(k+1)}=\frac{2k+2}{2k+1}\left(\frac{4k^2+2k-2}{4k^2+2k}\right)^{k-1}=\frac{(1-\frac{1}{k(2k+1)})^{k-1}}{1-\frac{1}{2k+2}}
\geq \frac{1-\frac{k-1}{k(2k+1)}}{1-\frac{1}{2k+2}}\geq 1.
\end{align*}
Hence for $k\geq 20$, we have $(1-\frac{3}{2k})^{k-2}\leq (1-\frac{3}{40})^{18}\leq 0.246$ and $(1-\frac{1}{2k})^{k-1}\leq (1-\frac{1}{40})^{19}\leq 0.619$. This completes the proof for (c).
For (d), since $m(k-m-1)\geq k-2$, we have $\frac{k-1}{k-m-1}\leq \frac{m(k-1)}{k-2}\leq 2m$, which implies the first inequality of (d).
Let $\gamma(m)=\sqrt{2m}(\frac{e^{1/3}}{\sqrt{2\pi}})^m$.
Then it suffices to show $\gamma(m)\leq 1$.
To see it, we have $\gamma(1)=0.7874<1$ by calculate directly.
Observe that for all $m\in \mathbb{N}^+$
\[
\frac{\gamma(m+1)}{\gamma(m)}=\sqrt{\frac{m+1}{m}}\frac{e^{1/3}}{\sqrt{2\pi}}\leq \frac{\sqrt{2}e^{1/3}}{\sqrt{2\pi}}=0.7874<1,
\]
so $\gamma(m)\leq \gamma(1) < 1$ holds for all positive integer $m$, completing the proof.
\end{proof}

We will also use Stirling's Formula to provide the following estimations on the factorials.
\begin{prop}\label{lem:Stirling}
For all $k\in \mathbb{N}^+$, $k!>\sqrt{2\pi}(\frac{k}{e})^{k}\sqrt{k}$. Moreover, we have
\[
\frac{k!}{\sqrt{2\pi}(\frac{k}{e})^k\sqrt{k}}\leq 1.01 \mbox{ for } k\geq 9
\mbox{, and }
\frac{k!}{\sqrt{2\pi}(\frac{k}{e})^k\sqrt{k}}\leq 1.05 \mbox{ for } k\geq 2.
\]
\end{prop}

\begin{proof}
Let $h(k) = \frac{k!}{(\frac{k}{e})^k\sqrt{k}}$. Then we have
$\frac{h(k+1)}{h(k)}=\frac{e(k+1)}{(k+1)(1+\frac{1}{k})^{k+\frac{1}{2}}}=\frac{e}{(1+\frac{1}{k})^{k+\frac{1}{2}}}< 1,$
thereby $h(k)$ is decreasing.
On the other hand, using Stirling's Formula, we see $\lim_{k\to\infty} h(k)=\sqrt{2\pi}$.
Hence $h(k)> \sqrt{2\pi}$ for all $k\in \mathbb{N}^+$,
which implies that $k!>\sqrt{2\pi}(\frac{k}{e})^{k}\sqrt{k}$.
For $k \geq 9$, we have $h(k) \leq h(9) \leq 1.01\sqrt{2\pi}$, which implies that $\frac{k!}{\sqrt{2\pi}(\frac{k}{e})^k\sqrt{k}} \leq 1.01$ for
$k\geq 9$. Similarly, we have $h(k)\leq h(2)\leq 1.05\sqrt{2\pi}$ for $k\geq 2$, indicating that $\frac{k!}{\sqrt{2\pi}(\frac{k}{e})^k\sqrt{k}} \leq 1.05$ for $k\geq 2$.
\end{proof}

Now we are ready to prove some numerical conclusions using these tools.

\begin{proof}[proof of Lemma \ref{lem:type2}]
Recall that $r\geq 3, k\geq 3$, $1\leq b_1 \leq \ldots \leq b_r$
and $k+1=b_1 + \ldots + b_r$.

Firstly, we prove item (A). 
In this case $b_1=1$. 
Suppose that $(V_1,\ldots, V_r)$ is the vertex partition of $H=\tilde{H}(n,(1,b_2,\ldots,b_r))$ such that $|V_i|= \frac{rb_i-1}{rk}n$ for $i\in [r]$.
Then by definition, we have 
\[E(\tilde{H}(n,(1,b_2,\ldots,b_r)))\subseteq \binom{V_1}{1}\times \binom{V_2\cup\ldots\cup V_r}{k-1}\cup \binom{V_2\cup\ldots\cup V_r}{k}.\]
Let $V_1',\ldots V_{r-1}'$ be a partition of $V_1$ such that $|V_i'|=\frac{n}{rk}$ for $1\leq i\leq r-1$ and let $V_r' = V_2\cup \ldots \cup V_r$.
Then $|V_r'| = \frac{r(k-1)+1}{rk}n$ and $H((V_1',\ldots, V_r'), (0,\ldots,0,k-1)) = \binom{V_1}{1}\times\binom{V_2\cup\ldots\cup V_r}{k-1}\cup \binom{V_2\cup\ldots\cup V_r}{k}$.
Hence
\[
\tilde{H}(n,(1,b_2,\ldots,b_r))\subseteq H((V_1',\ldots, V_r'), (0,\ldots,0,k-1))\cong H(n,(0,\ldots,0,k-1)), 
\]
finishing the proof of (A).

Now we prove (B). 
In this case $b_1\geq 2$ and $k= b_1 + \ldots + b_r - 1 \geq 2r - 1 \geq 5$. 
By the proof of Proposition \ref{prop:basic_ineq}, $h_1(k)=(1-\frac{1}{k})^{k-1}$ is decreasing over $\mathbb{N}^+$, hence
\[
f(k)\geq f_0(k) = 1 - (1-\frac{1}{k})^{k-1}
\geq 1- (1-\frac{1}{5})^{5-1} =0.5904.
\]
By the observation that $\frac{b_i}{rb_i-1}\leq \frac{b_1}{rb_1-1}$ for all $i\geq 2$, we have
\begin{equation}\label{equ:inequality_of_frac}
\frac{b_1-1}{rb_1-1} + \sum_{i=2}^r\frac{b_i}{rb_i-1}
\leq \frac{b_1-1}{rb_1-1} + (r-1)\frac{b_1}{rb_1-1}
=1.    
\end{equation}
Using inequality \eqref{equ:inequality_of_frac} and Proposition \ref{lem:Stirling}, we have
\begin{align*}
\tilde{g}_r(\vec{\mathbf{b}})
=~ &\binom{k}{b_1-1,b_2,\ldots,b_r}\cdot \prod_{i=2}^r\left(\frac{rb_i-1}{rk}\right)^{b_i}\cdot \left(\frac{rb_i-1}{rk}\right)^{b_1-1}\cdot\left(\frac{r(b_1-1)}{rb_1-1}+\sum_{i=2}^r\frac{rb_i}{rb_i-1}\right)\\
\leq~ & r\cdot \binom{k}{b_1-1,b_2,\ldots,b_r}\cdot \prod_{i=2}^r\left(\frac{rb_i-1}{rk}\right)^{b_i}\cdot \left(\frac{rb_i-1}{rk}\right)^{b_1-1}\\
\leq~ & 1.05r \cdot \frac{\sqrt{2\pi}\cdot k^{k+\frac{1}{2}}}{\sqrt{2\pi}\cdot (b_1-1)^{b_1-1+\frac{1}{2}} \prod_{i=2}^r \left(\sqrt{2\pi}\cdot  b_i^{b_i+\frac{1}{2}}\right)}\cdot \prod_{i=2}^r\left[\left(\frac{b_i}{k}\right)^{b_i}\cdot\left(1-\frac{1}{rb_i}\right)^{b_i}\right]\\
&\cdot\left(\frac{b_1-1}{k}\right)^{b_1-1}\cdot\left(1+\frac{r-1}{r(b_1-1)}\right)^{b_1-1}\\
\leq~ & 1.05r \cdot\sqrt{\frac{k}{(b_1-1)\cdot b_2\cdot\ldots\cdot b_r}}\left(\frac{1}{\sqrt{2\pi}}\right)^{r-1}\cdot\frac{\left(1+\frac{1}{r(b_1-1)}\right)^{(b_1-1)(r-1)}}{\prod_{i=2}^r\left(1+\frac{1}{rb_i}\right)^{b_i}}\\
\leq~ &  1.05r \cdot\sqrt{\frac{k}{(b_1-1)\cdot b_2\cdot\ldots\cdot b_r}}\left(\frac{1}{\sqrt{2\pi}}\right)^{r-1},
\end{align*}
where the third inequality holds because $1+kx\leq (1+x)^k$ and $(1+x)(1-x)=1-x^2\leq 1$ for $x\geq 0$ and $k\geq 1$ and the last inequality holds because $h_2(x)=(1+\frac{1}{x})^{x}$ is increasing for $x> 0$, implying that
$(1+\frac{1}{r(b_i-1)})^{b_i-1}\leq (1+\frac{1}{rb_i})^{b_i}$ for all $2\leq i\leq r$.

Noticing that $k + 1=b_1 + b_2 +\ldots + b_r$, $r\geq 3$ and $b_1, b_2,\ldots,b_r\geq 2$, we have $\frac{k}{(b_1-1)\cdot b_2 \cdot \ldots \cdot b_r}\leq \frac{5}{4}$ and it reaches the maximum value if and only if $r=3$, $k=5$ and $b_1=b_2=b_3=2$.
Hence
\[
\tilde{g}_r(\vec{\mathbf{b}})\leq 1.05\cdot \sqrt{\frac{5}{4}}\cdot r\left(\frac{1}{\sqrt{2\pi}}\right)^{r-1}.
\]
Let $h_3(r)=r\left(\frac{1}{\sqrt{2\pi}}\right)^{r-1}$. Then 
\[
\frac{h_3(r+1)}{h_3(r)} = \frac{r+1}{\sqrt{2\pi}\cdot r}\leq 1 \text{ for } r\geq 3.
\]
This implies that $h_3(r)$ is decreasing for $r\geq 3$.
Hence we have
\[
\tilde{g}_r(\vec{\mathbf{b}})\leq 1.05\cdot \sqrt{\frac{5}{4}}\cdot 3\cdot \left(\frac{1}{2\pi}\right)\leq 0.561 < 0.5904\leq f(k).
\]
This completes the proof of Lemma \ref{lem:type2}.

\end{proof}

\section{Proof of Lemma~\ref{lem:g<k}}\label{sec:g<k}
In this section, we give a detailed proof to show $f(k)>g_r(k)$ if one of the following cases holds:
\[
\mbox{(I). $r=2$ and $k\geq 20$; ~~~~ (II). $r\geq 3$ and $k\geq 10$; ~~~~ (III). $r\geq 6$ and $3\leq k\leq 9$.}
\]

In the remainder of this section, we establish Lemma~\ref{lem:g<k} by dividing it into three lemmas,
addressing cases (I), (II), and (III) respectively.
In each lemma, we demonstrate that $g_r(k) < f_0(k) \leq f(k)$.

We will use Propositions \ref{prop:basic_ineq} and \ref{lem:Stirling} to illustrate the following numerical conclusions.

\begin{lem}\label{lem:rk_ineq3}
Suppose that $r=2$ and $k\geq 20$ are positive integers. Then $g_r(k) < f(k)$.
\end{lem}
\begin{proof}
For $r=2$, by Lemma~\ref{lem:smallest_deg} we have
\[
g_2(k)
=\max_{\substack{0\leq a\leq b,\\ a+b=k-1}}g_r(a,b)
=\max_{\substack{0\leq a\leq b,\\ a+b=k-1}}\binom{k-1}{a}\left(\frac{2a+1}{2k}\right)^a\left(\frac{2b+1}{2k}\right)^b\left(1+\frac{a(2b+1)}{(2a+1)(b+1)}\right).
\]
For $k\geq 20$, using Proposition \ref{prop:basic_ineq} (c), we have
$f(k)\geq f_0(k)= 1-(1-\frac{1}{k})^{k-1}\geq 1-0.378=0.622.$
So it suffices to show that under the conditions $k\geq 20$, $a\leq b$ and $a+b=k-1$, we have $g_2(a,b)< 0.622$.
To see this, by Proposition \ref{prop:basic_ineq} (c), we see that under the assumption $k\geq 20$,
\begin{align*}
\mbox{if~} &a=0,~ g_2(a,b)=g_2(0,k-1)= \left(\frac{2k-1}{2k}\right)^{k-1}\leq 0.619< 0.622, \mbox{~and~}\\
\mbox{if~} &a=1,~ g_2(a,b)=g_2(1,k-2)=\frac{3(k-1)}{2k}\left(\frac{2k-3}{2k}\right)^{k-2}\left(1+\frac{2k-3}{3(k-1)}\right)\leq \frac{3}{2}\cdot 0.246\cdot \frac{5}{3}< 0.622.
\end{align*}
Now it remains to consider when $2\leq a\leq b$.
Using Proposition \ref{prop:basic_ineq} (c) and Proposition \ref{lem:Stirling}, we have
\begin{align*}
g_2(a,b)
&= \binom{k-1}{a}\left(\frac{2a}{2(k-1)}\right)^a\left(1+\frac{1}{2a}\right)^a\left(\frac{2b}{2(k-1)}\right)^b\left(1+\frac{1}{2b}\right)^b \left(1-\frac{1}{k}\right)^{k-1}\cdot \left(1+\frac{a(2b+1)}{(2a+1)(b+1)}\right)\\
&\leq \frac{1.01 (\frac{k-1}{e})^{k-1}\sqrt{2\pi (k-1)}}{(\frac{a}{e})^a\sqrt{2\pi a}\cdot (\frac{b}{e})^b\sqrt{2\pi b}}\cdot\frac{a^a}{(k-1)^a}\cdot\exp(\frac{a}{2a})\cdot\frac{b^b}{(k-1)^b}\cdot\exp(\frac{b}{2b})\cdot 0.378\cdot\left(1+\frac{2a}{2a+1}\right)\\
&=1.01\cdot 0.378\cdot e\cdot\frac{(k-1)^{k-1+\frac{1}{2}}}{(k-1)^{a+b}}\cdot \frac{a^a}{a^{a+\frac{1}{2}}}\cdot \frac{b^b}{b^{b+\frac{1}{2}}}\cdot (2\pi)^{-\frac{1}{2}}\cdot\frac{4a+1}{2a+1}\\
&=1.01\cdot 0.378\cdot e \cdot \sqrt{\frac{k-1}{2\pi ab}}\cdot \frac{4a+1}{2a+1},
\end{align*}
where the last equality holds because $k-1=a+b$.
If $a=2$, then this implies that
\[
g_2(a,b)\leq \sqrt{\frac{1}{2\pi}}\cdot 1.01\cdot 0.378 \cdot\sqrt{\frac{k-1}{2(k-3)}}\cdot e \cdot \frac{9}{5}
\leq
\sqrt{\frac{1}{2\pi}}\cdot 1.01\cdot 0.378 \cdot\sqrt{\frac{19}{34}}\cdot e \cdot \frac{9}{5}\leq 0.558 <0.622.
\]
So we may assume that $a\geq 3$. In this case, we have
\[
g_2(a,b)\leq \sqrt{\frac{1}{2\pi}}\cdot 1.01\cdot 0.378 \cdot\sqrt{\frac{k-1}{3(k-4)}}\cdot e \cdot 2
\leq \sqrt{\frac{1}{2\pi}}\cdot 1.01\cdot 0.378 \cdot\sqrt{\frac{19}{48}}\cdot e \cdot 2
\leq 0.521<0.622.
\]
Now for all integers $0\leq a\leq b$ and $a+b=k-1$, we establish $g_2(a,b)< f(k)$, finishing the proof.
\end{proof}

Next we prove the case (II) of Lemma~\ref{lem:g<k}.

\begin{lem}\label{lem:rk_ineq1}
Suppose that $r\geq 3$ and $k\geq 10$ are positive integers. Then
$g_r(k)< f(k)$.
\end{lem}
\begin{proof}
By Proposition \ref{prop:basic_ineq}, we have $f(k)\geq 1-1.06e^{-1}\geq 0.61$ for $k\geq 10$.
So it suffices to show that for all integers $0\leq a_1\leq \ldots \leq a_r$ with $a_1+\ldots+a_r=k-1$, we have
\[
g_r(a_1,\ldots,a_r)=\binom{k-1}{a_1,\ldots,a_r}\cdot \prod_{i=1}^r\left(\frac{ra_i+1}{rk}\right)^{a_i}\cdot \left(1+\frac{a_1}{ra_1+1}\sum_{i=2}^r\frac{ra_i+1}{a_i+1}\right)< 1-1.06e^{-1}.
\]

First we consider there exists some $1\leq \ell \leq r-1$ such that $a_1=\ldots=a_\ell=0$ and $a_{\ell+1},\ldots,a_r>0$.
Note that $\prod_{i=\ell+1}^r a_i$ is minimized by $(a_{\ell+1},\ldots, a_{r-1},a_r)=(1,\ldots, 1, k-r+\ell)$.
Since $a_1=0$, we have
\begin{align*}
g_r(a_1,\ldots,a_r)
&=\binom{k-1}{a_1,\ldots, a_r}\prod_{i=\ell+1}^r\left [
\left(\frac{ra_i}{r(k-1)}\right)^{a_i}\left(1+\frac{1}{ra_i}\right)^{a_i}\right ]\cdot \left(\frac{r(k-1)}{rk}\right)^{k-1}\\
&\leq\frac{1.01(\frac{k-1}{e})^{k-1}\sqrt{2\pi(k-1)}}{\prod_{i=\ell+1}^r(\frac{a_i}{e})^{a_i}\sqrt{2\pi a_i}}\cdot\prod_{i=\ell+1}^r\left [ \frac{a_i^{a_i}}{(k-1)^{a_i}}\exp\left(\frac{a_i}{ra_i}\right)\right]\cdot 1.06\cdot e^{-1}\\
&=(1.01\cdot 1.06)\cdot \sqrt{\frac{k-1}{\prod_{i=\ell+1}^r a_i}}\cdot \left(\frac{1}{2\pi}\right)^{\frac{r-\ell-1}{2}}\cdot\exp\left(-\frac{\ell}{r}\right)\\
&\leq(1.01\cdot 1.06)\cdot \sqrt{\frac{k-1}{(k-1)-(r-\ell-1)}}\cdot \left(\frac{1}{2\pi}\right)^{\frac{r-\ell-1}{2}}\cdot\exp\left(-\frac{\ell}{r}\right)\\
&\leq (1.01\cdot 1.06)\cdot \exp\left(-\frac{r-\ell-1}{r}\right)\exp\left(-\frac{\ell}{r}\right)\\
&=1.01\cdot 1.06\cdot\exp\left(-\frac{r-1}{r}\right)\leq 1.01\cdot 1.06\cdot\exp\left(-\frac{2}{3}\right) < 0.55 < 1-1.06e^{-1},
\end{align*}
where the first inequality holds because of Proposition \ref{prop:basic_ineq} (b) and Proposition \ref{lem:Stirling}, the second inequality holds since $\prod_{i=\ell+1}^r a_i\geq k-r+\ell$, and the third inequality holds because of Proposition \ref{prop:basic_ineq} (d) (by taking $m=r-\ell-1 \leq r-2$).\footnote{Note that if $r-\ell-1=0$, then the third inequality holds trivially.}
Hence, $g_r(a_1,\ldots,a_r) < f(k)$ whenever $a_1=0$.

Now we suppose that $a_1 > 0$. Then $r\leq k-1$ and $1+\sum_{i=2}^r\frac{a_1}{ra_1+1}\cdot\frac{ra_i+1}{a_i+1}\leq r$.
In this case,
\begin{align*}
g_r(a_1,\ldots,a_r)
\leq r\binom{k-1}{a_1,\ldots, a_r}\cdot\prod_{i=1}^r\left [
\left(\frac{ra_i}{r(k-1)}\right)^{a_i}\left(1+\frac{1}{ra_i}\right)^{a_i}\right ]\cdot \left(\frac{r(k-1)}{rk}\right)^{k-1}.
\end{align*}
Following the same discussion as above (e.g., taking $\ell=0$), we have
\begin{equation}\label{estimate_for_gr(1)}
g_r(a_1,\ldots,a_r)\leq(1.01\cdot 1.06\cdot r)\cdot\sqrt{\frac{k-1}{k-r}}\cdot\left(\frac{1}{2\pi}\right)^{\frac{r-1}{2}}.
\end{equation}
If $r=3$, then
$g_r(a_1,\ldots,a_r)\leq(1.01\cdot 1.06)\cdot\sqrt{\frac{9}{10-3}}\cdot\frac{1}{2\pi}\cdot 3 < 0.58 < 1-1.06e^{-1}\leq f(k)$.
If $r\geq 4$, then
\begin{align}\label{estimate_for_gr(2)}
g_r(a_1,\ldots,a_r)&\leq (1.01\cdot 1.06\cdot r)\cdot\sqrt{\frac{k-1}{k-r}}\cdot\left(\frac{1}{2\pi}\right)^{\frac{r-1}{2}}\leq 1.01\cdot 1.06\cdot r^{\frac{3}{2}}\cdot \left(\frac{1}{2\pi}\right)^{\frac{r-1}{2}},
\end{align}
where we use $\sqrt{\frac{k-1}{k-r}}\leq \sqrt{r}$ for $1\leq r\leq k-1$.
Let $h(r)=r^{\frac{3}{2}}(\frac{1}{2\pi})^{\frac{r-1}{2}}$. Then for all $r\geq 4$,
\begin{equation}\label{decrease_of_h}
\frac{h(r+1)}{h(r)}=\left(\frac{r+1}{r}\right)^{\frac{3}{2}}\sqrt{\frac{1}{2\pi}}
\leq \left(\frac{5}{4}\right)^{\frac{3}{2}}\sqrt{\frac{1}{2\pi}}\leq 1,
\end{equation}
implying that $h(r)$ is decreasing.
Then for $r\geq 4$, we can derive from \eqref{estimate_for_gr(2)} that
\begin{align*}
g_r(a_1,\ldots,a_r) &\leq 1.01\cdot 1.06\cdot r^{\frac{3}{2}}\cdot \left(\frac{1}{2\pi}\right)^{\frac{r-1}{2}}
\leq 1.01\cdot 1.06\cdot 4^{\frac{3}{2}}\cdot \left(\frac{1}{2\pi}\right)^{\frac{3}{2}}\leq 0.55 < 1-1.06e^{-1}\leq f(k),
\end{align*}
Now we have showed that $g_r(k)< f(k)$ for $r\geq 3$ and $k\geq 10$, completing the proof.
\end{proof}

Finally, we prove the last case (III) of Lemma~\ref{lem:g<k}, whose proof is a modification of the case (II).

\begin{lem}\label{lem:rk_ineq2}
Suppose that $r\geq 6$ and $k\geq 3$ are positive integers. Then $g_r(k)< f(k)$.
\end{lem}
\begin{proof}
Let $r\geq 6$ and $k\geq 3$.
By Proposition \ref{prop:basic_ineq} (a), we have $\left (1-\frac{1}{k}\right)^{k-1}\leq \frac49\leq 1.21\cdot e^{-1}$
and $f(k)\geq f_0(k)\geq \frac{5}{9} \geq 0.555$.
Also by Proposition~\ref{lem:Stirling}, we have $\frac{(k-1)!}{\sqrt{2\pi (k-1)}\left(\frac{k-1}{e}\right)^{k-1}}\leq 1.05.$
If there exists some $1\leq \ell\leq r-1$ such that $a_1 = \ldots = a_\ell =0$ and $a_{\ell+1}>0$,
then following the same discussion as in the proof of Lemma~\ref{lem:rk_ineq1}, we can derive that
\begin{equation}\label{estimate_for_gr(3)}
g_r(a_1,\ldots,a_r)\leq 1.05\cdot 1.21\cdot \exp\left (-\frac{r-1}{r}\right )
\leq 1.05\cdot 1.21\cdot \exp\left (-\frac{5}{6}\right )\leq 0.553<\frac59\leq f(k).
\end{equation}
Now we may assume that $a_1>0$. Then following the same arguments as for \eqref{estimate_for_gr(2)} and \eqref{decrease_of_h}, we derive
\begin{align*}
g_r(a_1,\ldots,a_r)&\leq 1.05\cdot 1.21 \cdot r^{\frac{3}{2}}\cdot \left(\frac{1}{2\pi}\right)^{\frac{r-1}{2}}\leq  1.05\cdot 1.21\cdot 6^{\frac{3}{2}}\cdot\left(\frac{1}{2\pi}\right)^{\frac{5}{2}}\leq 0.190<\frac{5}{9} \leq f(k),
\end{align*}
where we use the fact that $h(r)=r^{\frac{3}{2}}(\frac{1}{2\pi})^{\frac{r-1}{2}}$ is decreasing for $r\geq 6$.
Putting everything together, we see that $g_r(k)< f(k)$ holds whenever $r\geq 6$ and $k\geq 3$.
\end{proof}

We have completed the proof of Lemma~\ref{lem:g<k}. \qed

\section{Numerical support for Table~\ref{tab:case4} and Table~\ref{tab:case5}}\label{table}
Here we provide numerical support for Table~\ref{tab:case4} and Table~\ref{tab:case5}.

To achieve this, we present Mathematica Codes along with outcomes that can be downloaded from the following link:
\url{http://staff.ustc.edu.cn/~jiema/Color-biasPM-AppendixB.pdf}

\end{document}